\documentclass[11pt,leqno]{article}
\usepackage{amsthm,amsfonts,amssymb,amsmath}
\usepackage{epsfig}
\usepackage{graphicx}
\numberwithin{equation}{section}
\usepackage[thinlines]{easybmat}
\usepackage{epstopdf}
\usepackage{pstricks}
\def\eps{\varepsilon }

\newcommand\R{\mathbb R}

\def\eps{\varepsilon}

\newcommand\br{\begin{remark}}
\newcommand\er{\end{remark}}
\newcommand\bp{\begin{pmatrix}}
\newcommand\ep{\end{pmatrix}}
\newcommand\be{\begin{equation}}
\newcommand\ee{\end{equation}}
\newcommand\ba{\begin{equation}\begin{aligned}}
\newcommand\ea{\end{aligned}\end{equation}}

\newcommand{\bap}{\begin{app}}
\newcommand{\eap}{\end{app}}
\newcommand{\begs}{\begin{exams}}
\newcommand{\eegs}{\end{exams}}
\newcommand{\beg}{\begin{example}}
\newcommand{\eeg}{\end{exaplem}}
\newcommand{\bpr}{\begin{proposition}}
\newcommand{\epr}{\end{proposition}}
\newcommand{\bt}{\begin{theorem}}
\newcommand{\et}{\end{theorem}}
\newcommand{\bc}{\begin{corollary}}
\newcommand{\ec}{\end{corollary}}
\newcommand{\bl}{\begin{lemma}}
\newcommand{\el}{\end{lemma}}
\newcommand{\brs}{\begin{remarks}}
\newcommand{\ers}{\end{remarks}}

\newcommand{\sign}{{\text{\rm sgn }}}

\newtheorem{theo}{Theorem}[section]

\newtheorem{exams}[theo]{Examples}

\numberwithin{equation}{section}

\newcommand{\sgn}{\text{\rm sgn}}

\newtheorem{theorem}{Theorem}[section]
\newtheorem{proposition}[theorem]{Proposition}
\newtheorem{corollary}[theorem]{Corollary}
\newtheorem{lemma}[theorem]{Lemma}

\newtheorem{example}[theorem]{Example}
\newtheorem{remark}[theorem]{Remark}

\newcommand{\RM}{\mathbb{R}}

\newcommand{\CM}{\mathbb{C}}
\newcommand{\NM}{\mathbb{N}}

\newcommand{\cn}{\operatorname{cn}}
\newcommand{\dn}{\operatorname{dn}}

\title{Spectral stability of periodic wave trains of the\\ 
Korteweg-de Vries/Kuramoto-Sivashinsky equation\\
in the Korteweg-de Vries limit}

\author{\sc \small
Mathew A. Johnson\thanks{Department of Mathematics, University of Kansas, Lawrence, KS 66045; matjohn@math.ku.edu.}
~~~
Pascal Noble\thanks{Universit\'e de Lyon, Universit\'e Lyon I, Institut Camille Jordan, UMR CNRS 5208, 43 bd du 11 novembre 1918, F - 69622 Villeurbanne Cedex, France; noble@math.univ-lyon1.fr:
Research of P.N. is partially supported by the French ANR Project no.
ANR-09-JCJC-0103-01.}
~~~
L. Miguel Rodrigues\thanks{Universit\'e de Lyon, Universit\'e Lyon 1,
Institut Camille Jordan, UMR CNRS 5208, 43 bd du 11 novembre 1918,
F - 69622 Villeurbanne Cedex, France; rodrigues@math.univ-lyon1.fr:
Stay of M.R. in Bloomington was supported by French ANR project no. ANR-09-JCJC-0103-01
}
~~~
Kevin Zumbrun\thanks{Indiana University, Bloomington, IN 47405; kzumbrun@indiana.edu:
Research of K.Z. was partially supported under NSF grant no. DMS-0300487}
}
\begin{document}
\date{\today}
\maketitle

\begin{center}
{\bf 2010 MR Subject Classification}: 35B35, 35B10, 35Q53.
\end{center}

\begin{abstract}
We study the spectral stability of a family of periodic wave trains 
of the Korteweg-de Vries/Kuramoto-Sivashinsky equation
$ \partial_t v+v\partial_x v+\partial_x^3 v+\delta\left(\partial_x^2 v
+\partial_x^4 v\right)=0, $
$\delta>0$,
in the Korteweg-de Vries limit $\delta\to 0$,
a canonical limit describing small-amplitude weakly unstable thin film flow.
More precisely, we carry out a rigorous singular perturbation
analysis reducing the problem
to the evaluation for each Bloch parameter $\xi\in [0,2\pi]$
of certain elliptic integrals derived formally
(on an incomplete set of frequencies/Bloch parameters,
hence as necessary conditions for stability) 
and numerically evaluated by Bar and Nepomnyashchy \cite{BN},
thus obtaining, up to machine error, complete conclusions about stability.
The main technical difficulty is in treating the large-frequency and
small Bloch-parameter regimes not studied by Bar and Nepomnyashchy \cite{BN},
which requires
techniques rather different from classical Fenichel-type analysis.
The passage from small-$\delta$ to small-$\xi$
behavior is particularly interesting, using in an essential
way an analogy with hyperbolic relaxation
at the level of the Whitham modulation equations.
\end{abstract}
\newpage

\tableofcontents

\section{Introduction}\label{s:introduction}


\noindent
In this paper, we study the spectral stability of periodic wave trains of the
Korteweg-de Vries/Kuramoto-Sivashinsky (KdV-KS) equation
\begin{equation}\label{kdv-ks}
\displaystyle
\partial_t u+u\partial_x u+\partial_x^3 u+\delta\big(\partial_x^2 u+\partial_x^4 u\big)=0,\quad \forall t>0,\forall x\in\R,
\end{equation}
\noindent
with  $0<\delta\ll 1$.
When $\delta=0$, equation \eqref{kdv-ks} reduces to the well-studied Korteweg-de Vries (KdV) equation, which is an example
of a completely integrable infinite dimensional Hamiltonian system.  As such, the KdV equation is solvable by the inverse
scattering transform, and serves as a canonical integrable equation in mathematical physics and applied mathematics describing
weakly nonlinear dynamics of long one-dimensional waves propagating in a dispersive medium.

When $\delta>0$ on the other hand, equation \eqref{kdv-ks}
accounts for both dissipation and dispersion in the medium.
In particular, for $0<\delta\ll 1$ it is known to model
a thin layer of viscous fluid flowing down an incline, in which case it can be derived either from the shallow water equations
$$
\displaystyle
\partial_t h+\partial_x (hu)=0,\quad \partial_t (hu)+\partial_x(hu^2+\frac{h^2}{2F^2})=h-u^2+\nu\partial_x(h\partial_x u).
$$
as $F\to 2^+$
($F$ being the Froude number, with $F=2$ the
critical value above which steady constant-height flows are unstable)
or from the full Navier-Stokes equations if $0<R-R_c\ll 1$ ($R_c$ being the critical Reynolds number above which steady Nusselt flows are unstable) in the small amplitude/large scale regime: see \cite{Wi,YY} for more details.
%
For other values of $\delta$,
\eqref{kdv-ks} serves as a canonical model for pattern formation
that has been used to describe, variously, 
plasma instabilities, flame front propagation,
or turbulence/transition to chaos in reaction-diffusion systems
\cite{S1,S2,SM,K,KT}.

Here, our goal is to analyze the spectral stability of periodic traveling wave solutions of \eqref{kdv-ks}
with respect to small localized perturbations
in the singular limit  $\delta\to 0$.  In this limit the governing equation \eqref{kdv-ks} may be regarded as
a dissipative (singular) perturbation of the KdV equation, for which it is known that all periodic traveling waves
are spectrally stable to small localized perturbations; see \cite{BD,KSF,Sp}.  However, as the limiting KdV equation is 
time-reversible (Hamiltonian), 
this stability is of ``neutral'' (neither growing nor decaying) type,
and so
it is not immediately clear 
whether
the stability of these limiting
waves carries over to stability of ``nearby" waves in the
flow induced by \eqref{kdv-ks} for $|\delta|\ll 1$.  
Indeed, we shall see that, for different parameters,
neutrally stable periodic KdV waves may perturb to {\it either}
stable or unstable periodic KdV-KS waves, depending on the results of a rather
delicate perturbation analysis.

Our analysis, mathematically speaking, falls in the context
of perturbed integrable systems, a topic of independent interest.
In this regard, it seems worthwhile to mention that the proof of stability
of limiting KdV waves, and in particular the explicit determination of
eigenvalues and eigenfunctions of associated linearized operators
on which the present analysis is based, is itself a substantial
problem that remained for a long time unsettled.
Indeed, by an odd oincidence, both the original proof of spectral
stability in \cite{KSF,Sp}
and a more recent proof of spectral and linearized stability in \cite{BD}
(see also the restricted nonlinear stability result \cite{DK})
were accompanied by claims appearing at about
the same time of instability of these waves, an example of history repeating
itself and indirect indication of the difficulty of this problem.

However, our motivations for studying this problem come very much
from the physical applications to thin-film flow, and particularly
the interesting metastability phenomena described in \cite{PSU,BJRZ,BJNRZ3}
(see Section \ref{s:disc} below).
Interestingly, our resolution of the most difficult aspect of this problem,
the analysis of the small-Floquet number/small-$\delta$ regime,
is likewise motivated by the associated physics, in particular,
by the formal Whitham equations expected to govern long-wave perturbations
of KdV waves, and an extended relaxation-type system formally governing
the associated small-$\delta$ (KdV-KS) problem.

The identification of this structure, and the merging of integrable system
techniques with asymptotic ODE techniques introduced recently 
in, e.g., \cite{JZ2,PZ,HLZ,BHZ} (specifically, in our analysis
of frequencies $|\lambda|\in [C,C\delta^{-1}]$), 
we regard as interesting contributions
to the general theory that may be of use in related problems involving
perturbed integrable systems.
Our main contribution, though, is to the theory of thin-film flow, for
which the singular limit $\delta\to 0$
appears to be {\it the} canonical problem directing asymptotic behavior.
 
We begin by defining the notion of spectral stability 
of periodic waves of the $4$th-order parabolic system (KdV-KS),
following \cite{BJNRZ1}, as satisfaction of the following
collection of nondegeneracy and spectral assumptions:

\begin{itemize}
\item (H1) The map $H:\R^6\to \R^3$ taking $(X,b,c,q)\mapsto (u,u',u'')(X,b,c,q)-b$ is 
full rank at $(\bar X,\bar b,\bar c, \bar q)\in H^{-1}(\{0\})$ where $(u,u',u'')(\cdot;b,c,q)$ is the solution of
$$
\delta(u'''+u')+u''+(u-c)u'=q,\quad (u,u',u'')(0;b,c,q)=b.
$$

\item (D1) $\sigma_{L^2(\R)}(L)\subset\left\{\lambda\in\CM\,|\,\Re(\lambda)<0\right\}\cup\{0\}$, where
\[
L=-\delta\left(\partial_x^4+\partial_x^2\right)-\partial_x^3-\partial_x\left(\bar u-\bar c\right)
\]
denotes the linearized operator obtained by linearizing \eqref{kdv-ks} about $\bar u=u(.;\bar b,\bar c,\bar q)$.
\item (D2) $\sigma_{L^2_{per}(0,\bar X)}(L_\xi)\subset\{\lambda\in\CM\,|\,\Re\lambda\leq -\theta|\xi|^2\}$ for some $\theta>0$ and any $\xi\in[-\pi/\bar X, \pi/\bar X)$ where
\[
L_\xi[f]=-\delta\left((\partial_x+i\xi)^4f+(\partial_x+i\xi)^2f\right)-(\partial_x+i\xi)^3f-(\partial_x+i\xi)\left((\bar u-\bar c) f\right)
\]
denotes the associated Bloch operator with Bloch-frequency $\xi$.
\item (D3) $\lambda=0$ is an eigenvalue of the Bloch operator$L_0$ of algebraic multiplicity two.
\end{itemize}

Under assumptions $(H1),(D3)$, standard spectral perturbation theory implies the existence of two
eigenvalues $\lambda_j(\xi)\in\sigma(L_\xi)$ bifurcating from the $(\lambda,\xi)=(0,0)$ state
of the form $\lambda_j(\xi)=i\alpha_j\xi+o(\xi)$.
Assumption $(D1)$ ensures that 
$\alpha_j\in\RM$.
To ensure analyticity in $\xi$ of the critical curves $\lambda_j$, we assume further:
\begin{itemize}
\item (H2) The coefficients 
$\alpha_j$ are distinct.
\end{itemize}

The above definition of spectral stability
is justified by the results of \cite{BJNRZ1}, which state that,
under assumptions (H1),(H2) and (D1)-(D3), 
the underlying wave $\bar u$
is $L^1(\RM)\cap H^4(\RM)\to L^\infty(\RM)$ nonlinearly stable;
moreover, if $\tilde{u}$ is any other solution
of \eqref{kdv-ks} with data sufficiently
close to $\bar u$ in $L^1(\RM)\cap H^4(\RM)$, for some appropriately prescribed
$\psi$,,
the modulated solution $\tilde{u}(\cdot-\psi(\cdot,t),t)$ converges
to $\bar u$ in $L^p(\RM)$,  $p\in[2,\infty]$.

This is to be contrasted with the notion of spectral stability
of periodic waves of Hamiltonian systems, which, up to genericity
conditions analogous to (H1)-(H2) and (D3), amounts to the
condition that the associated linearized operator analogous to
$L$ have {\it purely imaginary spectrum}.
That is, in order that a (neutrally) stable periodic wave of (KdV)
perturb under small $\delta>0$ to a stable periodic wave of (KdV-KS),
its spectra must perturb from the imaginary axis into the 
stable (negative real part) complex half-plane


The main goal of this paper, therefore
is to 
establish by rigorous singular perturbation theory 
{\it a simple numerical condition guaranteeing
the existence of periodic traveling wave solutions of \eqref{kdv-ks} satisfying the above hypotheses (H1)-(H2) and (D1)-(D3)
for sufficiently small $\delta>0$}:
more precisely, determining whether the neutrally stable periodic solutions
of (KdV) perturb for small $\delta>0$ to stable or to unstable solutions
of (KdV-KS).

\br\label{stabrmk}
The methods used in \cite{BJNRZ1} to treat the dissipative case
$\delta>0$,
based on linearized decay estimates and variation of constants,
are quite different from those typically used to show stability in
the Hamiltonian case.  The latter are typically based on
{\it Arnolds method}, which consists of finding sufficiently many
additional constants of motion, or ``Casimirs,'' 
that the relative Hamiltonian becomes positive
or negative definite subject to these constraints
(hence controlling the norm of perturbations); that is,
additional constants of motion are used to effectively ``excise''
unstable (stable) eigenmodes of the second variation of the Hamiltonian.
This approach is used in \cite{DK} to show stability with respect to
$nX$-periodic perturbations for arbitrary $n\in\NM$, where $X$ denotes the period
of the underlying KdV wave train.
However, as $L^2$ spectra in the periodic case is purely essential, 
such unstable (stable) eigenmodes are uncountably many, and so 
it is unclear how to carry out this approach for general $L^2$ perturbations.
Indeed, to our knowledge, the problem of stability of periodic KdV waves
with respect to general $H^s$ perturbations remains open.
\er


The spectral stability of periodic wave-train solutions of \eqref{kdv-ks} itself has a long and interesting history
of numerical and formal investigations.
In \cite{CDK}, the authors studied numerically the spectral stability of periodic wave trains of
$$
\partial_t u+u\partial_x u+\gamma\partial_x^3 u+\partial_x^2 u+\partial_x^4 u=0,
$$
which is, up to a rescaling, equation \eqref{kdv-ks} with $\delta=\gamma^{-1}$ and showed the stabilizing effect of strong dispersion (large $\gamma$/small $\delta$s). As $\gamma$ is increased from $0$ to $8$ ($\delta\in(1/8, \infty)$), only one family of periodic waves of Kuramoto-Sivashinsky equation survives and its domain of stability becomes larger and larger and seems to ``converge'' to a finite range $(L_1, L_2)$ with $L_1=2\pi/0.74$ and $L_2\approx 2\pi/0.24$. In \cite{BN}, Bar and Nepomnyashchy studied formally the spectral stability of periodic wave trains of \eqref{kdv-ks} as $\delta\to 0$.
More precisely, for a fixed Bloch wave number $\xi$ the associated eigenvalues $\lambda(\xi,\delta)\in\sigma(L_\xi)$ are formally
expanded as
\begin{equation}\label{formalexpand}
\lambda(\delta,\xi)=\lambda_0(\xi)+\delta\lambda_1(\xi)+O(\delta^2),
\end{equation}
where $\lambda_0(\xi)\in\RM i$ is an eigenvalue associated to the stability of periodic waves of KdV, known explicitly (see \cite{BD}),
and $\lambda_1(\xi)$ %
is described in terms of elliptic integrals.
Then, the authors verified numerically, using high-precision computations
in MATHEMATICA (see Appendix B, \cite{BN}),
that $\max_{\xi\in[-\pi/\bar X, \pi/\bar X)}\Re\left(\lambda_1(\xi)\right)<0$,
consistent with stability,
on the band of periods
$\bar X\in(L_1,L_2)$
with $L_1\approx 8.49$ and $L_2\approx 26.17$, which are approximately the bounds found in \cite{CDK}.
Similar bounds were found numerically  in \cite{BJNRZ1}
by completely different, direct Evans function, methods,
with excellent agreement to those of \cite{BN}.

However, the study of Bar and Nepomnyashchy
is only formal and in particular,
as mentioned in \cite{BN}, it is not valid in the neighbourhood of the origin $(\lambda,\xi)=(0,0)$.
In particular, 
as we show in Section \ref{sec4},
the 
description \eqref{formalexpand} is
valid only for $0<\frac{\delta}{|\xi|}\ll 1$ hence,
for any given $\delta>0$ it is not possible to conclude 
from this expansion
spectral stability of an associated periodic
wave train of \eqref{kdv-ks}.
(The formal derivation of this expansion in \cite{BN} 
is valid for $|(\lambda,\xi)|_{\CM\times\RM}$ bounded from zero.)

Likewise, 
the numerical study in Section 2 of \cite{BJNRZ1}, 
which is not a singular perturbation
analysis, but rather a high-precision computation down to 
small
but positive $\delta$,\footnote{
Minimum value $\delta = .05$, as compared
to $\delta =.125$ in \cite{CDK}; see Table 3, \cite{BJNRZ1}.}
gives information about $\delta\to 0$ only
at finite scales, hence in effect omits an $O(\delta \times TOL)$ neighborhood
of the origin, where $TOL$ is the relative precision of the computation.
Thus, though very suggestive,
neither of these computations gives conclusive results about stability
in the $\delta \to 0$ limit, and, in particular,
{\it behavior on a $\delta$-neighborhood of the origin
is not (either formally or numerically) described.}

In this paper, we both make rigorous 
the formal singular perturbation
analysis that was done in \cite{BN} 
and extend it to the frequency regimes that were
omitted in \cite{BN}, completing
the study of the spectrum at the origin and in the high-frequency regime.
More precisely, we carry out a rigorous singular perturbation analysis
reducing the problem to the study of Bloch parameters $\xi\in [C\delta, 2\pi]$ 
(see Section \ref{sec3} for definition of Bloch parameter)
and eigenvalues $|\Re \lambda|\le C\delta$,
$|\Im \lambda|\le C$, $C>0$ sufficiently large,
on which the computations of \cite{BN} may
be justified by standard Fenichel-type theory.

The exclusion of high frequencies is accomplished by a standard
parabolic energy estimate restricting $|\Im \lambda|\le C\delta^{-3/4}$
followed by a second energy estimate on a reduced ``slow,"
or ``KdV," block
restricting $|\Im \lambda|\le C$;
see Lemma \ref{lemma1} and Proposition \ref{prop1} in Section \ref{compact} below.
For related singular perturbation
analyses using this technique of successive reduction
and estimation,
see for example \cite{MZ,PZ,Z,JZ2} and especially \cite{BHZ}, Section 4.
The treatment of small frequencies proceeds as usual by quite different
techniques involving rather the isolation of
``slow modes'' connected with formal modulation and large-time asymptotic
behavior.

At a technical level, this latter task appears quite daunting,
being a two-parameter
bifurcation problem emanating from a triple root $\lambda=0$
of the Evans function at $\xi=\delta=0$, where the Evans function
$E(\lambda, \xi,\delta)$ (defined in 
\eqref{evansdef} below)
is an analytic function whose roots $\lambda$ for
fixed $\delta>0$ and 
$\xi\in[-\pi/\bar X, \pi/\bar X)$
comprise the $L^2$-spectrum
of the linearized operator about the periodic solution.
However, using the special structure of the problem, we are able
to avoid the analysis of presumably complicated behavior
on the main ``transition regime''
$\Lambda:=\{(\delta,\xi)\ |\ C^{-1}\le |\xi|/\delta\le C\}$, $C\gg1$
and only examine the two limits $|\xi|/\delta\to 0, +\infty$ on which
the problem reduces to a pair of manageable {\it one-parameter bifurcation
problems} of familiar types.

Specifically, we show that (small) roots $\lambda$ of the Evans function cannot
cross the imaginary axis within $\Lambda$, so that stability properties
need only be assessed on the closure of $\Lambda$-complement, with
the results then propagating by continuity from the boundary
of $\Lambda$ to its interior.
This has the further implication that stability properties on the wedges
$|\xi|/\delta\ll 1$ and $|\xi|/\delta \gg1$
are linked (through $\Lambda$), and so it suffices to check stability
on the single wedge $|\xi|/\delta \gg 1$, where the analysis reduces
to computations carried out in
\cite{BN}.
Indeed, the situation is simpler still:
{\it stability on the entire
region $|\xi|, \delta\ll 1$ reduces by the above considerations to
validity of a certain ``subcharacteristic condition''} relating
characteristics of the Whitham modulation equations for KdV
(the limit $|\xi|/\delta \to \infty$)
and characteristics of a limiting reduced system as $\delta\to 0$
(the limit $|\xi|/\delta \to 0$).

As the above terminology suggests, there is a strong analogy in
the regime $|\xi|,\delta\ll 1$ to the
situation
of symmetric hyperbolic relaxation systems
and stability of constant solutions in the large time or small
relaxation parameter limit \cite{SK,Yo,Ze}, for which a similar
``noncrossing'' principle reduces the question of stability to
checking of Kawashima's genuine coupling condition, which in
simple cases reduces to the subcharacteristic condition that
characteristics of relaxation and relaxed system interlace.
Indeed, at the level of Whitham modulation equations, the limit
as $|\xi|/\delta \to \infty$ may be expressed as a relaxation from
the Whitham modulation equations for KdV to the Whitham equations
for fixed $\delta$ in the limit as $\delta \to 0$, a relation
which illuminates both the role/meaning of the subcharacteristic condition
and the relation between KdV and perturbed systems at the level
of asymptotic behavior.
These issues, which we regard as some of the most interesting
and important observations of the paper,
are discussed in Section \ref{whitham}.

The final outcome, and the main result of this paper, is that
{\it stability}- whether spectral, linear, or nonlinear-
{\it of periodic solutions of \eqref{kdv-ks} in the
KdV limit $\delta\to 0$ apart from the periods
$X_1\approx 8.49$ and $X_2\approx 26.17$ at which 
${\rm Ind}(X)=\max_{-\pi/X\le \xi\le \pi/X}\Re \lambda_1(\xi)$
precisely vanishes,
is determined by $\hat \sigma(X):=\sgn\, {\rm Ind}(X)$,  }
or, equivalently, the values of the elliptic integrals derived by
Bar and Nepomnyashchy in \cite{BN}, with $\hat \sigma(X)<0$ corresponding
to stability.

\subsection{Discussion and open problems}\label{s:disc}

As we have emphasized above, our analysis, though convincing we feel,
does not constitute a numerical proof, but 
rather
a ``numerical
demonstration,'' in the sense that the computations of \cite{BN}
on which we ultimately rely for evaluation of $\hat \sigma(X)$ are carried out
with high precision and great numerical care, but not with
interval arithmetic in a manner yielding guaranteed accuracy.
However, there is no reason that such an analysis could not be carried out-
we point for example to the computations of \cite{M} in the related context
of stability of radial KdV-KS waves- and, given the fundamental nature of
the problem, this seems an important open problem for further investigation.

Indeed, numerical proof of stability or instability for arbitrary
nonzero values of $\delta$, verifying the numerical conclusions
of \cite{BJNRZ1}, or of Evans computations in general,
though considerably more involved, seems also feasible,
and another important direction for future investigations.

The particular limit $\delta \to 0$ studied here has special
importance, we find, as a canonical limit that serves (as discussed
at the beginning of the introduction) as an organizing
center for other situations/types of models as well, and
it has indeed been much studied; see, for example, \cite{EMR,BN,PSU},
and references therein.
As discussed in \cite{PSU,BJNRZ3,BJRZ}, it is also prototypical of
the interesting and somewhat surprising behavior of inclined thin
film flows that solutions often organize time-asymptotically into
arrays of ``near-solitary wave'' pulses, despite the fact that
individual solitary waves, since their endstates necessarily induce
unstable essential spectrum,\footnote{A straightforward Fourier
transform computation reveals that all constant solutions are unstable.}
are clearly {\it unstable}.

To pursue the analogy between modulational behavior and
solutions of hyperbolic-parabolic conservation or balance laws
that has emerged in \cite{OZ,Se,BJNRZ1,BJNRZ2},
etc., and, indeed, through the earlier studies of \cite{FST} or
the still earlier work of Whitham \cite{W},
we feel that the KdV limit $\delta\to 0$ of \eqref{kdv-ks}
plays a role for small-amplitude periodic inclined thin film flow analogous
to that played by Burgers equation for small-amplitude shock waves
of general systems of hyperbolic--parabolic conservation laws,
and the current analysis a role analogous to that of Goodman's analysis
in \cite{Go1,Go2} of spectral stability of general small-amplitude
shock waves by singular perturbation of Burgers shocks.\footnote{
See also the related \cite{PZ,FreS}, 
more in the spirit of the present analysis.}

The difference from the shock wave case is that, whereas, up to Galilean
and scaling invariances, the Burgers shock profile is unique, there exists up
to invariances a one-parameter family of periodic waves of KdV, indexed by
the period $X$, of which only a certain range are stable.  Moreover, whereas
the Burgers shock profile is described by a simple $\tanh$ function, 
periodic KdV waves are described by a more involved parametrization
involving elliptic functions.  Thus, the study of periodic waves is inherently
more complicated, simply by virtue of the number of cases that must be considered, and the complexity of the waves involved.
Indeed, in contrast to the essentially geometric proof of Goodman for shock waves, we here find it necessary to use in essential ways certain exact computations
coming from the integrability/inverse scattering 
formalism of the underlying KdV equation. 

\medskip

{\bf Plan of the paper}.
In Section
%
\ref{sec2}, we compute an expansion of periodic waves of KdV-KS in the limit $\delta\to 0$ by 
using Fenichel singular perturbation theory. In Section \ref{sec3}, we justify rigorously the 
formal spectral analysis in \cite{BN}: 
we provide, first, a priori estimates on the size of unstable eigenvalue and show  that they 
are necessarily of order $O(1)$ as $\delta\to 0$. Then we compute an expansion of both the Evans 
function  and eigenvalues with respect to $\delta$.  This analysis holds true except in a 
neighborhood of the origin from which spectral curves bifurcate. In Section \ref{sec4}, we 
compute the spectral curves in the neighborhood of the origin and show that spectral 
stability is related to subcharacteristic conditions for a Whitham's modulation system of relaxation type.

\section{\label{sec2}Expansion of periodic traveling-waves in the KdV limit}

For $0<\delta\ll 1$,
equation \eqref{kdv-ks} is a singular perturbation of the Korteweg-de Vries equation
\begin{equation}\label{kdv}
\partial_t u+u\partial_x u+\partial_x^3 u=0,
\end{equation}
where the periodic traveling wave solutions may be described with the help of the Jacobi elliptic functions.
In \cite{EMR}, periodic traveling wave solutions of \eqref{kdv-ks} are proved to be $\delta$-close to periodic
traveling wave solutions of \eqref{kdv} and, furthermore, an expansion of these solutions with respect to $\delta$ is found.
We begin our analysis by briefly recalling the details of this expansion.
Notice that \eqref{kdv-ks} admits traveling wave solutions of the form $u(x,t)=U(x-ct)$ provided the profile
$U$ satisfies the equation
$$
\displaystyle
(U-c)U'+U'''+\delta(U''+U'''')=0,
$$
where here $'$ denotes differentiation with respect to the traveling variable $\omega=x-ct$.
Due to the conservative nature of \eqref{kdv-ks} this profile equation may be integrated once yielding
\begin{equation}\label{eq:pr}
\displaystyle
\frac{U^2}{2}-cU+U''+\delta(U'+U''')=q,
\end{equation}
where $q\in\RM$ is a constant of integration.  
By introducing $x=U, y=U'$ and $z=U''+U$, we may write \eqref{eq:pr} as the equivalent first order system
\begin{equation}
\label{l1}
\displaystyle
x'=y,\quad y'=z-x,\quad \delta z'=-z+q+(c+1)x-\frac{x^2}{2}.
\end{equation}
\noindent
Setting $\delta=0$ in \eqref{l1} yields the slow system
\begin{equation}
\displaystyle
z=q+(c+1)x-\frac{x^2}{2},\quad x'=y,\quad y'=q+cx-\frac{x^2}{2},
\end{equation}
which is equivalent to the planar, integrable system governing the traveling wave profiles for the KdV equation \eqref{kdv}.
Utilizing the well-known Fenichel theorems, we are able to justify the reduction and continue
the resulting KdV profiles for $0<\delta\ll 1$.
To this end, we define
$$
\displaystyle
M_0=\left\{(x,y,z)\in\mathbb{R}^3\ \middle|\ z=q+(c+1)x-\frac{x^2}{2}=:f_0(x)\right\},
$$
and recognize this as the slow manifold associated to (\ref{l1}).
It is readily checked that this manifold is normally hyperbolic attractive, and so a standard
application of the 
Fenichel theorems yields the following proposition.

\begin{proposition}
For $\delta>0$ sufficiently small, there
exists a slow manifold $M_{\delta}$ invariant under the flow of (\ref{l1}) that is written as
$$
\begin{array}{lll}
\displaystyle
M_{\delta}=\left\{(x,y,z)\in\mathbb{R}^3\ \middle|\ z=f_{\delta}(x,y)\right\},\quad f_{\delta}(x,y)=f_0(x)+\delta f_1(x,y)+\delta^2 f_2(x,y)+O(\delta^3),\\
\displaystyle
f_1(x,y)=xy-(c+1)y,\quad f_2(x,y)=-y^2+(c+1-x)(q+cx-\frac{x^2}{2}).
\end{array}
$$
\end{proposition}
\noindent
The expansion of $f_{\delta}$ is obtained by inserting $z=f_{\delta}(x,y)$ into (\ref{l1}) and identifying the powers in $\delta$. Then by plugging this expansion into (\ref{l1})$_2$, one finds the reduced {\it planar} system:
\begin{equation}\label{plan_sys}
\displaystyle
x'=y,\quad y'=q+cx-\frac{x^2}{2}+\delta(xy-(c+1)y)+O(\delta^2),
\end{equation}
\noindent
or equivalently the scalar equation
\begin{equation}\label{plan_eq}
 \displaystyle
 x''=q+cx-\frac{x^2}{2}+\delta (x-c-1)x'+O(\delta^2).
\end{equation}

Now, we seek an asymptotic expansion of the solutions of \eqref{plan_sys} in the limit $\delta\to 0$.
An easy way of doing these computations to any order with respect to $\delta$ is to follow the formal computations in \cite{BN},
which are now justified here with Fenichel's theorems.
To begin, notice that when $\delta=0$ the periodic solutions $x=x_0$ with wave speed $c=c_0$ of \eqref{plan_eq} agree with those of the
KdV equation \eqref{kdv}, which are given explicitly by 
\begin{equation}\label{kdvsoln}
x_0(\omega;\phi,\kappa,k,u_0)=u_0+12 k^2\kappa^2 \cn^2\big(\kappa(\omega+\phi),k\big),\quad c_0=u_0+8\kappa^2k^2-4\kappa^2,
\end{equation}
where $\cn(\cdot,k)$ is the Jacobi elliptic cosine function
with elliptic modulus $k\in[0, 1)$ and $\kappa>0,u_0,\phi$ are arbitrary real constants related to the
Lie point symmetries of \eqref{kdv}; see \cite{BD}.
Thus, the set of periodic traveling wave solutions of \eqref{kdv} forms a four dimensional manifold ($3$ dimensional up to translations)
parameterized by $u_0$, $\kappa$, $k$, and $\phi$.
Note that such solutions are $2K(k)/\kappa$ periodic, where $K(k)$ is the complete elliptic integral of the first kind.

\begin{remark}
The parameterization of the periodic traveling wave solutions of the KdV equation given in \eqref{kdvsoln} is consistent
with the calculations in \cite{BD} where the authors verify the spectral stability of such solutions to localized perturbations
using the complete integrability of the governing equation.  However, this parameterization is not the same as that given
in \cite{BN}, whose numerical results our analysis ultimately relies on.  Indeed, in \cite{BN} the periodic traveling
wave solutions of \eqref{kdv} are given (up to rescaling\footnote{In \cite{BN}, the authors consider the KdV equation in
the form $\tilde u_t+6\tilde u \tilde u_x+\tilde u_{xxx}=0$,
which is equivalent to \eqref{kdv} via the simple rescaling $\tilde u\mapsto \frac{1}{6}u$.}) as
\[
X_{\textrm{BN}}(\omega;\omega_0,q,k)=\frac{q^2K(k)^2}{3\pi^2}\left(\dn^2\left(\frac{(\omega-\omega_0) qK(k)}{\pi},k\right)-\frac{E(k)}{K(k)}\right),
\]
where $\dn(\cdot,k)$ denotes the Jacobi dnoidal function with elliptic modulus $k\in[0,1)$, and $K(k)$ and $E(k)$ denote
the complete elliptic integrals of the first and second kind, respectively.  Nevertheless, using the identity
\[
k^2\cn^2(x,k)=\dn^2(x,k)-(1-k^2)
\]
we can rewrite \eqref{kdvsoln} as
\[
x_0(\omega)=12\kappa^2\left(\dn^2(\kappa(\omega+\phi),k)+\frac{u_0}{12\kappa^2}-(1-k^2)\right),
\]
which, upon setting $\kappa=\frac{qK(k)}{\pi}$, $\phi=-\omega_0$, and choosing $u_0$ so that
\[
\frac{u_0}{12\kappa^2}-(1-k^2)=-\frac{E(k)}{K(k)},
\]
we see that $x_0(\omega)=X_{\textrm{BN}}(\omega)$.
Thus, there is no loss of generality in choosing one parameterization over
the other.  Furthermore, the numerical results of \cite{BN} carry over directly to the cnoidal wave parameterization
chosen here.
\end{remark}

Next, we consider the case $0<\delta\ll 1$.  To begin we seek conditions guaranteeing that periodic traveling wave solutions
of \eqref{kdv-ks} exist for sufficiently small $\delta$.  
Multiplying both sides by $x'$ and rearranging, 
we find that equation \eqref{plan_eq} may be written as
\begin{equation}
\displaystyle
\frac{d}{d\omega}\left(\frac{(x')^2}{2}+\frac{x^3}{6}-c\frac{x^2}{2}-qx\right)=\delta\big(x(x')^2-(c+1)(x')^2\big)+O(\delta^2),
\end{equation}
hence a
necessary condition for the existence of a $L$-periodic solution to \eqref{plan_eq} is
\begin{equation}\label{selec0}
\displaystyle
\int_0^L \left(x_0(\omega)(x_0'(\omega))^2-(c+1)(x_0'(\omega))^2\right)d\omega=0.
\end{equation}
By a straightforward computation using integration by parts and \eqref{plan_eq}, \eqref{selec0} can be simplified to the selection principle
\begin{equation}\label{selec}
\displaystyle
\int_0^L (x_0''(\omega))^2d\xi=\int_0^L (x_0'(\omega))^2d\xi,
\end{equation}
or, equivalently,
\begin{equation}\label{selecf}
\displaystyle
\kappa^2=\frac{\displaystyle\int_0^{2K(k)}[(cn^2)']^2(y)dy}{\displaystyle\int_0^{2K(k)}[(cn^2)'']^2(y)dy}=:F^2(k).
\end{equation}
\noindent
Using the implicit function theorem, one can show that if \eqref{selecf} is satisfied, there exists a periodic solution $x^{\delta}$ of \eqref{plan_eq} which is $\delta$ close to $x_0$. As a result, we obtain a $3$-dimensional set of periodic solutions to (\ref{kdv-ks}) parametrized by $u_0,\phi$ and either $k$ or $\kappa$. Note that the limit $\kappa\to 0$ (i.e. $k\to 1$) corresponds to a solitary wave and $\kappa\to 1$ (i.e. $k\to 0$) corresponds to small amplitude solutions (or equivalently  to the onset of the Hopf bifurcation branch).

The above observations lead us to the following proposition.

\begin{proposition}[\cite{EMR}]\label{p:kdvsolnexpand}
As $\delta\to 0$, the periodic traveling waves $U(\omega),\;\omega=x-ct$, solutions of \eqref{kdv-ks} expand (up to translations) as
\begin{equation}\label{kdvsolnexpand}
\left\{\begin{aligned}
U(\omega)&=U_0(\kappa \omega,u_0,k,\kappa)+\delta U_1(\omega)+\delta^2 U_2(\omega)+O(\delta^3),\\
c&=c_0(u_0,k,\kappa)+\delta^2 c_2+O(\delta^3),
\end{aligned}\right.
\end{equation}
where $U_0,c_0$ are defined as
$$
\displaystyle
U_0(y,u_0,k, \kappa)=u_0+3k\left(\frac{\kappa K(k)}{\pi}\right)^2\cn^2\left(\frac{ K(k)}{\pi}y,k\right),\quad c_0=u_0+(2k-1)\left(\frac{\kappa K(k)}{\pi}\right)^2,
$$
and $\kappa$
is determined from $k$ via the
selection principle $\kappa=\mathcal{G}(k)$ with
$$
\displaystyle
\left(\frac{K(k)\mathcal{G}(k)}{\pi}\right)^2=
   \frac{7}{20}\frac{2(k^4-k^2+1)E(k)-(1-k^2)(2-k^2)K(k)}{(-2+3k^2+3k^4-2k^6)E(k)+(k^6+k^4-4k^2+2)K(k)}.
$$
Moreover the functions $(U_i)_{i=1,2}$ are (respectively odd and even) solutions of the linear equations
\[
\mathcal{L}_0[U_1]+\kappa U_0''+\kappa^3 U_0''''=0,\quad
\mathcal{L}_0[U_2]+\left(\frac{U_1^2}{2}-c_2 U_0\right)'+\kappa U_1''+\kappa^3 U_1''''=0,
\]
where $\mathcal{L}_0:=\kappa^2 \partial_x^3+\partial_x\left((U_0-c_0).\right)$ is a closed linear operator acting on $L^2_{\rm per}(0,2K(k))$
with densely defined domain $H^3_{\rm per}(0,2K(k))$.
\end{proposition}

\begin{proof}
The explicit expansions above are determined as follows.  After rescaling, continuing the $2K(k)/\kappa$-periodic wave trains of \eqref{kdv}
to $0<\delta\ll 1$ is equivalent to searching for $2K(k)$-periodic solutions of
%
\begin{equation}\label{eq:prk}
\displaystyle
(U-c)U'+\kappa^2 U'''+\delta\Big(\kappa U''+\kappa^3U''''\Big)=0
\end{equation}
for $\delta>0$ sufficiently small.
We expand $c$, $U$  in the limit $\delta\to 0$ as
$$
\displaystyle
c=c_0+\delta c_1+O(\delta^2),\quad U=U_0(\xi)+\delta U_1(\xi)+O(\delta^2).
$$
\noindent
with $U_0(\omega)=x_0(\omega,\kappa,k,u_0)$ as defined in \eqref{kdvsoln}.  
Notice that, up to order
$O(1)$,
equation \eqref{eq:prk} is satisfied for {\it all} $u_0,k,\kappa$, i.e. there is no selection of a particular wave train.
Now, identifying the $O(\delta)$ terms in \eqref{eq:prk} yields the equation
\begin{equation}\label{dl1}
\displaystyle
\kappa^2 U_1'''+\big((U_0-c_0)U_1\big)'-c_1U_0'+\kappa U_0''+\kappa^3 U_0'''' = 0 .
\end{equation}
The linear operator $\displaystyle \mathcal{L}_0[x]=\kappa^2 x'''+\left((U_0-c_0)x\right)'$, defined on $H_{per}^3(0, 2K(k))$, is Fredholm of index $0$ and $(1, U_0)$ span the kernel of its adjoint (see \cite{BrJ,JZB} for more details). Then one can readily deduce that equation \eqref{dl1} has a periodic solution provided that the following compatibility condition is satisfied, $\displaystyle \langle(U'_0)^2\rangle=\kappa^2\langle(U''_0)^2\rangle$ which is precisely the selection criterion \eqref{selecf}. In order to determine $c_1$, one has to consider higher order corrections to $x_0$: in fact,  $c_1$ is determined through a solvability condition on the equation for $x_2$. This yields $c_1=0$ (see \cite{EMR} for more details).
\end{proof}

As a consequence,  we have obtained a two dimensional manifold of (asymptotic) periodic solutions (identified when coinciding up to translation) para\-metrized by $u_0\in\mathbb{R}$ and wave number $\kappa$ (or alternatively the parameter $k\in[0, 1)$). Note that the limit $\kappa\to 0$ (i.e. $k\to 1$) corresponds to a solitary wave and $\kappa\to 1$ (i.e. $k\to 0$) corresponds to small amplitude solutions.

\section{\label{sec3} Stability with respect to high
frequency perturbations }

In this section, we begin our study of the spectral stability of periodic traveling waves of \eqref{kdv-ks} in the limit
$\delta\to 0^+$.  
Denote by $U(\cdot)$ such a  
$X$-periodic
where
for notational convenience, we have dropped the $\delta$ dependence of both $U$ and $c$.
Linearizing \eqref{kdv-ks} about $U$ in the co-moving frame $(x-ct,t)=(\omega,t)$ leads to the linear evolution equation
\[
\partial_t v-Lv=0
\]
governing the perturbation $v$ of $U$, where $L$ denotes the differential operator with 
$X$-{\it periodic}
coefficients
\[
Lv=-\big((U-c)v\big)'-v'''-\delta(v''+v'''').
\]

In the literature, there are many choices for the  class of perturbations considered, each of which corresponds
to a different domain for the above linear operator.  Here, we are interested in perturbations of $U$ which
are {\it spatially localized}, hence we require that $v(\cdot,t)\in L^2(\RM)$ for each $t\geq 0$.
Seeking separated solutions of the form $v(x,t)=e^{\lambda t}v(x)$ then leads to the spectral ODE problem
\begin{equation}\label{sppb}
Lv=\lambda v,
\end{equation}
$v\in L^2(\RM)$.

To characterize the spectrum of the operator $L$, considered here as a densely
defined operator on $L^2(\RM)$, we note that as the coefficients
of $L$ are 
$X$-periodic
functions of $x$, Floquet theory implies that the
spectrum of $L$ 
is purely continuous and that $\lambda\in\sigma(L)$ if and only if the spectral problem
\eqref{sppb} has an $L^\infty(\RM)$ eigenfunction of the form
\be\label{e:ansatz}
v(x;\lambda,\xi)=e^{i\xi x}w(x;\lambda,\xi)
\ee
for some $\xi\in[-\pi/X,\pi/X)$ and 
$w(\cdot)\in L^2_{\rm per}(0,X)$
Following \cite{G,S1,S2}, we find that substituting the ansatz \eqref{e:ansatz}
into \eqref{sppb} leads one to consider the one-parameter family of Bloch operators
$\{L_\xi\}_{\xi\in[-\pi/X,\pi/X)}$ acting on 
$L^2_{\rm per}([0,X])$
via
\be\label{e:Lxi}
\left(L_\xi w\right)(x):=e^{-i\xi x}L\Big[e^{i\xi \cdot}w(\cdot)\Big](x).
\ee
Since the Bloch operators have compactly embedded domains 
$H^4_{\rm per}([0,X])$ in $L^2_{\rm per}([0,X])$,
their spectrum consists entirely of discrete eigenvalues which, furthermore, depend continuously on the
Bloch parameter $\xi$.  It follows 
by these standard considerations that
\[
\sigma_{L^2(\RM)}(L)=\bigcup_{\xi\in[-\pi/X,\pi/X)}\sigma_{L^2_{\rm per}([0,X])}\left(L_\xi\right);
\]
see \cite{G} for details. As a result, the spectrum of $L$ may be decomposed into countably many curves $\lambda(\xi)$
such that $\lambda(\xi)\in\sigma(L_\xi)$ for $\xi\in[-\pi/X,\pi/X)$.

The spectra $\lambda$ of the Bloch operators $L_\xi$ may be characterized
as the zero set for fixed $\xi,\delta$ of the {\it Evans function}
\be\label{evansdef}
E(\lambda,\xi,\delta)={\rm det}\left(R(X,\lambda,\delta)
-e^{i\xi X}Id_{\mathbb{C}^3}\right),
\ee
where $R(.,\lambda)$ denotes
the resolvent (or monodromy) matrix associated to the linearized 
equation \eqref{sppb},
that is, the solution of
$(L-\lambda)R(\cdot, \lambda)=0$ and $R(0,\lambda)={\rm Id}$.
Thus, the spectra of $L$ consists of the union of zeros $\lambda$
as all values of $\xi\in \RM$ are swept out.
By analytic dependence on parameters of solutions of ODE, $R$, and
thus $E$, depend analytically on all parameters for $\delta>0$.

In the following, we will first prove that possible unstable eigenvalues are order $O(1)+iO(\delta^{-3/4})$ by using
a standard parabolic energy estimate.  By a bootstrap argument based on an approximate
diagonalisation of the first order differential system associated to \eqref {sppb},
we show that possible unstable eigenvalues are $O(1)$ which implies that they are necessarily of order $O(\delta)+iO(1)$.
We then provide an expansion 
in $\delta$
of the Evans function 
as $\delta\to 0$
in a bounded box close to the
imaginary axis with the help of a Fenichel-type procedure and an iterative
scheme based on the exact resolvent matrix associated to the linearized KdV equations.

\subsection{Boundedness
of unstable eigenvalues as $\delta\to 0$}\label{compact}

In this section, we bound the region in the unstable half plane $\Re(\lambda)\geq 0$ where the unstable
essential spectrum of the linearized operator $L$ may lie in the limit $\delta\to 0$.
Throughout, we use the notation 
$\|u\|^2=\int_{0}^X |u(x)|^2dx$.
We begin by proving the following lemma, verifying
that the unstable spectra is $O(\delta^{-3/4})$ for $\delta$ sufficiently small.

\begin{lemma}\label{lemma1}
There exist constants $C_1,C_2>0$ such that, for all $\delta\in(0, 1]$, the operator $L$ has no
$L^\infty(\RM)$ eigenvalues with 
$\Re(\lambda)\geq C_1$ 
or ($\Re(\lambda)\geq0$ and $\delta^{3/4}|\Im(\lambda)|\geq C_2$).
\end{lemma}
\begin{proof}
Suppose that $\lambda$ is an $L^\infty(\RM)$ eigenvalue of $L$ 
 and let $v$ be a corresponding eigenfunction.
Multiplying equation \eqref{sppb} by $\bar{v}$ and integrating over one period,
we obtain
\begin{equation}\label{spbd}
\displaystyle
\lambda\|v\|^2-\int_{0}^{X}\big((U-c)v+v''\big)\bar{v}'dx+\delta\big(\|v''\|^2-\|v'\|^2\big)=0.
\end{equation}
Identifying the real and imaginary parts yields the 
system of equations:
\begin{equation}
\begin{array}{ll}
\displaystyle
\Re(\lambda)\|v\|^2+\frac{1}{2}\int_{0}^{X}U'|v|^2dx+\delta\big(\|v''\|^2-\|v'\|^2\big)=0,\\
\displaystyle
|\Im(\lambda)|\|v\|^2\leq \|(U-c)\|_{\infty}\|v\|\|v'\|+\|v'\|\|v''\|.
\end{array}
\end{equation}
Here, we have used the fact that, by \eqref{e:ansatz}, $v(x+X)=e^{i\xi X}v(x)$ so that $|v|$ is $X$-periodic.
Next, using the Sobolev estimate $\|v'\|^2\leq C\|v\|^2/2+\|v''\|^2/(2C)$, valid for any $C>0$, into the first equation yields
the bound
\begin{equation}\label{rebd}
\displaystyle
\Re(\lambda)\|v\|^2+\delta\left(1-\frac{1}{2C}\right)\|v''\|^2\leq \frac{1}{2}\left(\|U'\|_{\infty}+\delta C\right)\|v\|^2,\quad C>0.
\end{equation}
Letting $C=1/2$ then yields 
$$
\displaystyle
\Re(\lambda)\leq \frac{1}{2}\left(\|U'\|_{\infty}+\frac{\delta}{2}\right),
$$
which verifies the stated bound on the real part of $\lambda$.

Suppose now $\Re(\lambda)\geq 0$.
%
Using again the Sobolev estimate $\|v'\|\leq\|v\|^{1/2}\|v''\|^{1/2}$, the
imaginary part of $\lambda$ can be bounded as 
$$
\displaystyle
|\Im(\lambda)|\|v\|^2\leq \|U-c\|_{\infty}\|v\|^{3/2}\|v''\|^{1/2}+\|v\|^{1/2}\|v''\|^{3/2}.
$$
Furthermore, $\|v''\|$ can be controlled by
\[
\delta\|v''\|^2\leq \big(\|U'\|_{\infty}+\delta\big)\|v\|^2,
\]
which follows by setting $C=1$ in \eqref{rebd} and recalling that $\Re(\lambda)\geq 0$ by hypothesis.
Thus, setting $K^2=\|U'\|_{\infty}+\delta$ we deduce that
\[
|\Im(\lambda)|\leq \|U-c\|_{\infty}K^{1/2}\delta^{-1/4}+K^{3/2}\delta^{-3/4},
\]
which completes the proof.
\end{proof}

\begin{remark}\label{r:extend1}
By slight modification, the estimates in Lemma \ref{lemma1} can be extended into the stable spectrum.  Indeed,
if $C_3>0$ then adding $C_3\|v\|^2$ to both sides of the bound \eqref{rebd} yields the estimate
$$
\displaystyle
(\Re(\lambda)+C_3)\|v\|^2+\delta(1-\frac{1}{2C})\|v''\|^2\leq \frac{1}{2}(\|U'\|_{\infty}+2C_3+\delta C)\|v\|^2.
$$
Thus, as long as $\Re(\lambda)+C_3\geq 0$ we can repeat the proof on the estimates of imaginary parts to conclude
$$
\displaystyle
|\Im(\lambda)|\leq K_3^{1/2}\delta^{-1/4}+K_3\delta^{-3/4},
$$
where $K_3=\|U'\|_{\infty}+2C_3+\delta$.
\end{remark}

Next, we bootstrap the estimates in Lemma \ref{lemma1} to provide a second energy estimate on the reduced ``slow," or ``KdV,"
block of the spectral problem \eqref{sppb}.  This yields a sharper estimate on the modulus of the possibly unstable spectrum,
in particular proving that unstable spectra must lie in a compact region in the complex plane.  Notice that this
result relies heavily on the fact that the corresponding spectral problem for the linearized KdV equation about a
cnoidal wave \eqref{kdvsoln} has been explicitly solved in \cite{BD,Sp}.

\begin{proposition}\label{prop1}
There exist constants $C_1,C_2>0$ such that, for all $\delta\in(0, 1]$, the operator $L$ has no
$L^\infty(\RM)$ eigenvalues with $\Re(\lambda)\geq 0$
when $\displaystyle \Re(\lambda)\geq C_1\delta$ or $|\Im(\lambda)|\geq C_2$.
\end{proposition}

\begin{proof}
The proof is done in two steps: first, we show that if $\lambda$ is an $L^\infty(\RM)$ eigenvalue of $L$
with $\Re(\lambda)\geq 0$ and corresponding eigenfunction $v$,
then there exists $C_2>0$ such that $|\Im(\lambda)|\leq C_2$.  The estimate on $\Re(\lambda)$ will then easily follow.
To begin, let $(v,\lambda)$ be an $L^\infty(\RM)$-eigenpair of \eqref{sppb} with $\Re(\lambda)\geq 0$
and set $x=v$, $y=v'$, $z=v''+v$, $w=z'$, and $s=c+1$,
so that \eqref{sppb} may be written as the first order system
\begin{equation}\label{evans1}
\displaystyle
x'=y,\quad y'=z-x,\quad z'=w,\quad \delta w'=-w-(U'+\lambda)x-(U-s)y.
\end{equation}
We first apply a Fenichel-type procedure and introduce $w_1=w+(U'+\lambda)x+(U-s)y$, noting then  that $w_1$ satisfies
$$
\displaystyle
\delta w_1'=-w_1+\delta\Big(U''x+(2U'+\lambda)y+(U-s)(z-x)\Big).
$$
We further introduce $w_2=w_1-\delta\Big(U''x+(2U'+\lambda)y+(U-s)(z-x)\Big)$ so that $w_2$ satisfies the equation
\begin{eqnarray}
\displaystyle
\delta w_2'&=&-\big(1+\delta^2(U-s)\big)w_2-\delta^2\big((U'''-(U-s)(U'+\lambda-\delta U'')\big)x\nonumber\\
\displaystyle
&&-\delta^2\Big((3U''-(U-s)^2+\delta(U-s)(2U'+\lambda))y+(3U'+\lambda+\delta(U-s)^2)(z-x)\Big).\nonumber
\end{eqnarray}
Now, by Lemma \ref{lemma1} we know that necessarily one has $\Re(\lambda)+\delta^{3/4}|\Im(\lambda)|\leq C$
for some constant $C>0$.  It follows 
that $\lambda\delta=o(1)$ as $\delta\to 0$, hence we may
rewrite system \eqref{evans1} as
\begin{equation}\label{evans2}
\begin{aligned}
x'&=y,\qquad y'=z-x,\\
z'&=w_2-(U'+\lambda)x-(U-s)y+\delta\Big(U''x+(2U'+\lambda)y+(U-s)(z-x)\Big),\\
\delta w_2'&=-w_2-\delta^2\lambda\Big((z-x)-(U-s)x\Big)+O(\delta^2(|x|+|y|+|z|+|w_2|)).
\end{aligned}
\end{equation}
Next, we remove $w_2$ from the equation in $z$ by introducing the variable $z_*=z+\delta w_2$,
in terms of which
\eqref{evans2} reads
\begin{equation}\label{evans2_1}
\begin{aligned}
x'&=y,\qquad y'=z_*-\delta w_2-x,\\
z_*'&=-(U'+\lambda)x-(U-s)y+\delta\Big(U''x+(2U'+\lambda)y+(U-s)(z_*-x)\Big)\\
&\quad\quad-\delta^2\lambda\big((z_*-x)-(U-s)x\big)+O(\delta^2(|x|+|y|+|z_*|+|w_2|)),\\
\delta w_2'&=-w_2-\delta^2\lambda\Big((\bar{z}-x)-(U-s)x\Big)+O(\delta^2(|x|+|y|+|z_*|+|w_2|)).
\end{aligned}
\end{equation}
We further introduce the variables $\bar{y}=y-\delta^2 w_2$, $\bar{x}=x$ and $\bar{z}=z_*-\bar{x}$. The system \eqref{evans2_1} then reads
\begin{equation}\label{evans3}
\begin{aligned}
\bar{x}'&=\bar{y}+O(\delta^2(|\bar x|+|\bar y|+|\bar{z}|+|w_2|)),\quad \bar{y}'=\bar{z}+O(\delta^2(|\bar x|+|\bar y|+|\bar{z}|+|w_2|)),\\
\bar{z}'&=-(U'+\lambda)\bar{x}-(U-s)\bar{y}+\delta\Big(U''\bar x+(2U'+\lambda)\bar y+(U-s)\bar{z}\Big)\\
&\quad\quad-\delta^2\lambda\big(\bar{z}-(U-s)\bar{x}\big)+O(\delta^2(|\bar x|+|\bar y|+|\bar{z}|+|w_2|)),\\
\delta w_2'&=-w_2-\delta^2\lambda\Big(\bar{z}-(U-s)\bar{x}\Big)+O(\delta^2(|\bar x|+|\bar y|+|\bar{z}|+|w_2|)).
\end{aligned}
\end{equation}
In particular, by direct comparison with \eqref{kdv},
we recognize the $(\bar x,\bar y, \bar z)$ equations in \eqref{evans3}
as simply the KdV equation plus an $O(\delta)$ corrector.

The above calculations motivate us to make a reduction to the ``KdV block" of the spectral problem \eqref{sppb}.
More precisely, 
recalling \eqref{kdvsolnexpand}, 
we write the differential system
(\ref{evans3}$_{1}$, \ref{evans3}$_{2}$, \ref{evans3}$_{3}$) on
$\bar X=(\bar x,\bar y,\bar z)^T$ as
\begin{equation}\label{evans-kdv}
\displaystyle
\bar X'=\Big(A_{0}+\delta \big(A_1+\lambda A_2\big)+\lambda\delta^2A_3+O(\delta^2)\Big)\bar X+O(\delta^2|w_2|),
\end{equation}
where
\begin{equation}\label{kdvsystem}
A_{0}=\left(\begin{array}{ccc} 0 & 1 & 0\\
                                                    0 & 0 & 1\\
                                                    -(U_0'+\lambda) & -(U_0-c_0) & 0
                                                    \end{array}\right)
\end{equation}
denotes the coefficient matrix for the linearized KdV equation, 
and $A_1$, $A_2$, and $A_3$ are defined as
\[
A_1=\left(\begin{array}{ccc} 0 & 0 & 0\\
          0 & 0 & 0\\
          -U_1'+U'' &-U_1+ 2U' & U-s\end{array}
    \right),
A_2=\left(\begin{array}{ccc}
          0 & 0 & 0\\
          0 & 0 & 0\\
          0 & 1 & 0
          \end{array}\right),
A_3=\left(\begin{array}{ccc}
          0 & 0 & 0\\
          0 & 0 & 0\\
          (U-s) & 0 & -1
          \end{array}\right).
\]

In order to analyze \eqref{evans-kdv} for $0<\delta\ll 1$, we recall that
in \cite{BD} the complete integrability of \eqref{kdv} was used to determined a basis of
solutions of $X'=A_{0}X$, at least when $\lambda\neq 0$,
which corresponds to linearized KdV equation about the periodic wave train 
$U_0$.
Specifically, such a basis $(V_i)_{i=1,2,3}$ is defined as $V_i=(\hat{u}_i, \hat{u}_i', \hat{u}_i'')$ with $\hat{u}_i$ given by
\[
\hat{u}_i(\omega,\lambda)=\left(1-\frac{U_0'}{3\lambda}\right)\textrm{exp}\left(\int_0^{\omega} \frac{\lambda dy}{U_0(y)/3-c+\eta_i}\right),
\]
and $\eta_i$ are solutions of the polynomial equation
\begin{equation}\label{etaeqn}
(\eta-4\xi_1)(\eta-4\xi_2)(\eta-4\xi_3)=\lambda^2,
\end{equation}
where $\xi_1=k^2-1, \xi_2=2k^2-1, \xi_3=k^2$.  In order to deal with the limit
$|\lambda|\to\infty$, we introduce the diagonal matrix $D(\lambda)$ with
\[
D_{ii}(\lambda)=-\left\langle \frac{\lambda}{U_0/3-c+\eta_i}\right\rangle,
\]
where $\langle g(\cdot)\rangle$ denotes the average of the function $g$ over a spatial period of $U$,
and write a resolvent
matrix for $X'=A_{0}X$ as
\[
R(\lambda,\omega)=P(\lambda,\omega)e^{D(\lambda)\omega},
\]
where $P(\lambda,\omega)=(\bar{V}_1,\bar{V}_2,\bar{V}_3)(\lambda,\omega)$ is the matrix function
with columns being given by the vector valued functions
$\displaystyle\bar{V}_{k,i}(\lambda,\omega)=e^{-D_{kk}(\lambda)\omega}\partial_\omega^{(i-1)}\hat{u}_k(\omega), i=1,2,3$.
Next we make the {\it periodic} change of variable $\bar X(\omega)=P(\lambda,\omega) Y(\omega)$,
which is nothing but the classical change of variable in Floquet's theorem.  In terms of $Y$,  system \eqref{evans-kdv} expands as
\begin{equation}\label{evans-kdv1}
\displaystyle
Y'=\left(D(\lambda)+\delta P^{-1}\Big( A_1+\lambda A_2+\lambda\delta A_3+O(\delta)\Big)P\right)Y+O(\delta^2\|P^{-1}\||w_2|)
\end{equation}
as $|\lambda|\to\infty$.

We now analyze the individual terms in \eqref{evans-kdv1} more closely.  To this end, first notice that as $|\lambda|\to\infty$
the eigenfunctions associated to the linearized KdV equation expand as
\[
\hat{u}_i(\omega,\lambda)=\left(1+O\left(|\lambda|^{-1/3}\right)\right)e^{D_{ii}(\lambda)\omega}.
\]
It follows that as $|\lambda|\to\infty$ the matrix $P$ defined above expands as
\[
P=\left(\begin{array}{ccc} 1 & 1 & 1\\
                           \Lambda_1 & \Lambda_2 & \Lambda_3\\
                           \Lambda_1^2 & \Lambda_2^2 & \Lambda_3^2
\end{array}\right)(1+O(|\lambda|^{-1/3}),
\]
where 
$\Lambda_i:=D_{ii}(\lambda)$.  
Thus, by a straightforward calculation we see that
as $|\lambda|\to\infty$ we have the estimates $\|P(\lambda,\cdot)\|_{L^\infty(\RM)}=O(|\lambda|^{2/3})$ and
\[
\|P^{-1}(\lambda,\cdot)\|_{L^\infty(\RM)}=O(1),\quad \|P^{-1}A_1P\|_{L^\infty(\RM,d\omega)}=O(1),
\quad \|P^{-1}A_3P\|_{L^\infty(\RM,d\omega)}=O(1)
\]
hence, using the fact that $|\lambda|\delta^{3/4}=O(1)$, equation \eqref{evans-kdv1} can be rewritten as
\begin{equation}\label{evans-kdv2}
\displaystyle
Y'=\left(D(\lambda)+\delta(\lambda P^{-1}A_2P+O(1))+O(\delta^{5/4})\right)Y+O(\delta^2|w_2| ).
\end{equation}
Finally, with a near-identity change of variables
of the form $\widetilde Y=\left(Id+O(\delta|\lambda|^{1/3})\right)Y$
one can remove the non-diagonal part of $\delta(\lambda P^{-1}A_2P+O(1))$ up to $O(\delta^{5/4})$ so that \eqref{evans-kdv2} reads
\begin{equation}\label{evans-kdv3}
\displaystyle
\tilde{Y}'=\left(D(\lambda)+\delta{\rm diag}(\lambda P^{-1}A_2P+O(1))+O(\delta^{5/4})\right)\tilde{Y}+O(\delta^2|w_2|).
\end{equation}

Next, define the diagonal matrix $\Gamma(\lambda):=D(\lambda)+\delta{\rm diag}(\lambda P^{-1}A_2P+O(1))$ with diagonal entries
\begin{align*}
\displaystyle
\Gamma(\lambda)_{11}&=\Lambda_1+\delta\left(\lambda\frac{\Lambda_1(\Lambda_3-\Lambda_2)}{\Delta}+O(|\lambda|^{1/3})\right),\nonumber\\
\displaystyle
\Gamma(\lambda)_{22}&=\Lambda_2+\delta\left(\lambda\frac{\Lambda_2(\Lambda_1-\Lambda_3)}{\Delta}+O(|\lambda|^{1/3})\right),\nonumber\\
\displaystyle
\Gamma(\lambda)_{33}&=\Lambda_3+\delta\left(\lambda\frac{\Lambda_3(\Lambda_2-\Lambda_1)}{\Delta}+O(|\lambda|^{1/3})\right),
\end{align*}
where $\Delta:=(\Lambda_2-\Lambda_1)(\Lambda_3-\Lambda_1)(\Lambda_3-\Lambda_2)$.  From \eqref{etaeqn} it follows
that $\eta_i=O(|\lambda|^{2/3})$ as $|\lambda|\to\infty$, from which we see $\Lambda_i(\lambda)=O(|\lambda|^{1/3})$
in this limit.  Introducing the
polar coordinates $\lambda=|\lambda|e^{i(\pi/2-\theta)}$, and noting that
$\Re(\lambda)=O(1)$ by Lemma \ref{lemma1},
we find that $\theta=O(|\lambda|^{-1})$ as $|\lambda|\to\infty$.
Directly expanding the $D_{ii}(\lambda)$, we have
\[
\Lambda_1=|\lambda|^{1/3}e^{i(\pi/2-\theta/3)}+O(\lambda^{-1/3}),\quad \Lambda_2=j\Lambda_1+O(1),\quad \Lambda_3=j^2\Lambda_1+O(1),
\]
where $j=e^{2\pi i/3}$ denotes the principal third root of unity so that, in particular, we have the estimates
\begin{equation}\label{Lambdabds}
\Re(\Lambda_2)=\frac{\sqrt{3}}{2}|\lambda|^{1/3}+O(1),\quad \Re(\Lambda_3)=-\frac{\sqrt{3}}{2}|\lambda|^{1/3}+O(1)
\end{equation}
as $|\lambda|\to\infty$.

With the above preparations, we are now in a position to
perform the necessary energy estimates.
Indeed, under the condition 
$\tilde Y(x+X)=e^{i\gamma}\tilde Y(x)$ and $w_2(x+X)=e^{i\gamma} w_2(x)$
and recalling that $\|P(\lambda,\cdot)\|_{L^\infty(\RM)}=O(|\lambda|^{2/3})$, it follows from \eqref{evans3} that
\begin{equation}\label{w2bd}
\|w_2\|\leq C|\lambda|\delta^2\|P(\lambda,.)\|_{\infty}\big(\|\tilde x\|+\|\tilde y\|+\|\tilde z\|\big)
\leq C\delta^{3/4}\big(\|\tilde x\|+\|\tilde y\|+\|\tilde z\|\big),
\end{equation}
where here we set $\tilde Y=(\tilde x,\tilde y,\tilde z)^T$.  Similarly, using the bounds in \eqref{Lambdabds},
it follows from \eqref{evans-kdv3} that
\begin{equation}\label{yzbd}
\|\tilde{y}\|+\|\tilde{z}\|\leq C\frac{\delta^{5/4}}{|\lambda|^{1/3}}\|\tilde{x}\|.
\end{equation}
Inserting the bounds \eqref{w2bd} and \eqref{yzbd} into the $\tilde{x}$ equation in \eqref{evans-kdv3} and recalling
that the function $\tilde{x}$ must be uniformly bounded on $\RM$ as a function of $\omega$, we find necessarily
that $\Re(\Gamma(\lambda)_{11})=O(\delta^{5/4})$ as $|\lambda|\to\infty$, i.e. we have
\[
\Re(\Lambda_1)+\delta\left(\frac{|\lambda|^{2/3}}{3}+O(|\lambda|^{1/3})\right)=O(\delta^{5/4})
\]
which, as $|\lambda|\to\infty$, reduces to
\begin{equation}\label{evans-kdvf}
\displaystyle
0\leq \frac{\Re(\lambda)}{|\lambda|^{2/3}}(1+O(|\lambda|^{-1/3}))+\delta\left(|\lambda|^{2/3}+O(|\lambda|^{1/3})\right)\leq C\delta^{5/4}.
\end{equation}
Since we have assumed $\Re(\lambda)\geq 0$ it immediately follows that $|\lambda|$ must indeed be bounded.  More precisely,
we deduce that there exists $C_2$ and $\delta_1>0$ such that for  all $0<\delta<\delta_1$,
the operator $L$ has no unstable eigenvalues $\lambda$ on $L^\infty(\RM)$ such that
$|\lambda|>C_2$.  As we have already verified in Lemma \ref{lemma1} that $\Re(\lambda)$
is necessarily bounded, we obtain a uniform bound on $|\Im(\lambda)|$.
Moreover, it is then easy to show, by using \eqref{evans-kdvf}, that, necessarily,
possible unstable eigenvalues satisfy $0\leq\Re(\lambda)\leq C\delta$ for some
constant $C>0$, and the proposition is proved.\end{proof}

\begin{remark}\label{r:extend2}
As discussed in Remark \ref{r:extend1}, the estimate $|\Im(\lambda)|=O(\delta^{-3/4})$ is actually valid so long
as $\Re(\lambda)=O(1)$.  Thus, by repeating the argument of Proposition \ref{prop1}, one can prove that
for any $C>0$ there exists $M,\delta_1>0$ such that if $0\leq\delta\leq \delta_1$ and $|\lambda|\geq M$,
then there are no eigenvalues $\lambda$ such that $\Re(\lambda)\geq -C\delta$.
\end{remark}

As a result of Proposition \ref{prop1} and Remark \ref{r:extend2}, we have proved the following corollary.

\begin{corollary}\label{cor_eig}
Given any constant $C>0$, there exist constants $M,\delta_1>0$ such that for $0\leq\delta<\delta_1$ we have
\[
\sigma_{L^2(\mathbb{R})}(L)\subset\left\{\lambda\in\mathbb{C}\,|\,\Re(\lambda)\leq -C\delta \right\}\cup\left\{\lambda\in\mathbb{C}\,|\,|\Re(\lambda)|\leq C\delta,\:|\Im(\lambda)|\leq M\right\}.
\]
\end{corollary}

In summary, we have restricted the location of the
unstable part of the $L^2(\RM)$-spectrum of the linearized operator $L$
to a compact subset of $\mathbb{C}$,
uniformly for $\delta$ sufficiently small.
Our next goal is to prove convergence, for a fixed $\xi$, of the eigenvalues of the Bloch operator
$L_\xi$ to the eigenvalues of the linearized KdV equation as $\delta\to 0$. 
This is accomplished in the next
section through the use of the periodic Evans function.

\br\label{pzrmk}
The structure of the argument of Proposition \ref{prop1}
may be recognized as somewhat similar to those of arguments
used in \cite{JZ2,PZ,HLZ,BHZ} to treat other delicate limits
in asymptotic ODE.
A new aspect here is the incorporation of detailed estimates
on the limiting system afforded by complete integrability of (KdV),
which appear to be crucial in obtaining the final
result.
\er

\mathversion{bold}
\subsection{Expansion of the Evans function as $\delta\to 0$}
\mathversion{normal}

In this section, we provide an expansion of both the Evans function and eigenvalues in the
vicinity of the imaginary axis where {\it all} the eigenvalues are 
located at $\delta=0$ (this
is the spectral stability result of \cite{BD,Sp}).  
To this end, we will use the basis of solutions
constructed in \cite{BD} to build an approximation of the resolvent matrix associated to
the full spectral problem \eqref{sppb}.  This leads us to the following result.

\begin{proposition}\label{p:evansexpand}
On any compact set $\lambda \in K\subset \mathbb{C}$, the Evans function 
\eqref{evansdef} of the spectral problem \eqref{sppb} expands,
up to a nonvanishing analytic factor, as
 \begin{equation}\label{e:evans}
 \begin{array}{ll}
 \displaystyle
 E(\lambda,\xi,\delta)=E_{kdv}(\lambda,\xi)+\delta E_1(\lambda,\xi)+O(\delta^2),\quad 0<\delta\ll 1,
\end{array}
\end{equation}
with $\displaystyle E_{kdv}(\lambda,\xi)={\rm det}\left(R_{kdv}(X,\lambda)-e^{i\xi X}Id_{\mathbb{C}^3}\right),$
$R_{kdv}(.,\lambda)$ being the resolvent matrix associated to the linearized (KdV) equation.
As a consequence, for each fixed Bloch wave number $\xi\in[-\pi/X,\pi/X)$ and $\delta$ sufficiently small,
if $(\lambda_{\delta}(\xi))_{\delta>0}\in K$ is an associated eigenvalue of $L_\xi$ then
$\lambda_{\delta}(\xi)$ converges to $\lambda_0(\xi)$, an eigenvalue of the linearized KdV equation, as $\delta\to 0$.
\end{proposition}

\begin{proof}
First, we carry out a Fenichel-type computation on the spectral
problem \eqref{sppb} up to $O(\delta^3)$, noting that by Corollary \ref{cor_eig}
the $L^\infty(\RM)$ eigenvalues of the operator $L$ are uniformly bounded in $\CM$.
Recall that in the proof of Proposition \ref{prop1} the spectral problem \eqref{sppb} was transformed into system \eqref{evans3}:
\begin{equation}\label{evans4}
\begin{aligned}
\bar x'&=\bar y+O(\delta^2|w_2|),\quad \bar y'=\bar z+O(\delta^2(|\bar x|+|\bar y|+|\bar z|+|w_2|)),\\
\bar z'&=-(U'+\lambda)\bar x-(U-c)\bar y+\delta\left(U''\bar x+(2U'+\lambda)\bar y+(U-s)\bar z\right)\\
&\quad\quad+O(\delta^2(|\bar x|+|\bar y|+|\bar z|+|w_2|)),\\
\delta w_2'&=-w_2+O(\delta^2(|\bar x|+|\bar y|+|\bar z|+|w_2|)).
\end{aligned}
\end{equation}
Introducing $Y=(x,y,z)$, we can thus write \eqref{evans4} as
\begin{equation}\label{evans5}
\displaystyle
Y'=A(\delta,\lambda)Y+\delta^2 w_2\,F,\quad \delta w_2'=-(1+O(\delta^2))w_2+\delta^2 G^TY,
\end{equation}
where $F,G\in C^0_{per}((0,X);\mathbb{C}^3)$
and $A(\delta,\lambda)=A_{0}(\lambda)+O(\delta)$ where $A_0$ is given
in \eqref{kdvsystem}. Further, we denote by $R(\cdot,\lambda,\delta)$ the resolvent matrix associated
to $Y'=AY$. It is a clear consequence of the regularity of the flow associated to this latter
differential system that $R(\cdot,\lambda,\delta)$ is regular with respect to $(\lambda,\delta)$
and expands as $R(\cdot,\lambda,\delta)=R_{kdv}(\cdot,\lambda)+O(\delta)$ where $R_{kdv}$ is the
resolvent matrix of the linearized KdV equation $Y'=A_0Y$ satisfying the initial condition
$R_{kdv}(0,\lambda)=Id_{\mathbb{C}^3}$.
In order to simplify the notation in the forthcoming calculations,
we now drop the $(\lambda,\delta)$ dependence of resolvent matrices.

Next, we seek to construct a basis of solutions of \eqref{evans5} valid for $0<\delta\ll 1$.
To this end, notice that by Duhamel's formula the system \eqref{evans5} can be equivalently written as
\begin{equation*}
\begin{array}{ll}
\displaystyle
Y(\omega)=R(\omega)Y(0)+\delta^2\int_0^{\omega}w_2(\eta)R(\omega)R^{-1}(\eta)F(\eta)d\eta,\\
\displaystyle
w_2(\omega)=\exp\left(-\int_0^{\omega}\frac{1+O(\delta^2)}{\delta}dq\right)w_2(0)
   +\delta\int_0^{\omega}\exp\left(-\int_\eta^{\omega}\frac{1+O(\delta^2)}{\delta}dq\right)G^T(\eta)Y(\eta)d\eta
\end{array}
\end{equation*}
where here $\frac{1+O(\delta^2)}{\delta}$ denotes a fixed analytic function of $q$ of the specified order: for
definiteness, we denote this function by $\mu(q)$.
As a first step, we build a set of $3$ eigenvectors which are continuations of the eigenvectors
of the linearized KdV equation. For that purpose, we set $w_2(0)=0$ and write $Y$ as
\begin{equation}\label{fix_point}
\displaystyle
Y(\omega)=R(\omega)Y(0)+\delta^3\int_0^\omega \int_0^\eta
  \exp\left(-\int_\zeta^\eta\mu(q)dq\right)G^T(\zeta)Y(\zeta)R(\omega)R^{-1}(\eta)F(\eta)d\zeta d\eta.
\end{equation}
By applying a fixed point argument in 
$L^\infty_{\rm per}([0,X])$ 
to \eqref{fix_point}, we 
find a set of three eigenvectors
$(Y_i,w_{2,i})_{i=1,2,3}$ of \eqref{evans5} given by $Y_i=R(\xi)e_i+O(\delta^3)$ with
$e_{i,j}=\delta_{i,j}$ and $w_{2,i}=O(\delta)$.
To find a fourth linearly independent eigenvector of \eqref{evans5}, we seek
a solution $(Y,w_2)=(Y_4,w_{2,4})$ such that
\begin{equation}\label{w24}
w_2=\exp\left(\int_{\omega}^{X}\mu(q)dq\right)
    \left(1+\delta\int_0^\omega \exp\left(-\int_\eta^{X}\mu(q)dq\right)G^T(\eta)Y(\eta)d\eta\right);
\end{equation}
in particular, notice then that $w_2(0)\neq 0$.
Choosing $Y(0)$ then so that $Y(X)=0$ gives
\begin{equation}\label{Yeqn}
\begin{aligned}
&\exp\left(-\int_\omega^{X}\frac{1+O(\delta^2)}{\delta}dq\right)Y(\omega)=
-\delta^2\int_{\omega}^{X} \exp\left(-\int_\eta^\omega\mu(q)dq\right)R(\omega)R^{-1}(\eta)F(\eta)d\zeta d\eta\\
&\quad\quad\quad\quad\quad\quad
+\delta\int_0^{\eta}\exp\left(-\int_\zeta^{X}\mu(q)dq\right)G(\zeta)Y(\zeta)R(\omega)R^{-1}(\eta)F(\eta)d\zeta d\eta.
\end{aligned}
\end{equation}
We then apply a fixed point argument in weighted space 
$e^{-\int_\xi^{X}\frac{1+O(\delta^2)}{\delta}} L^{\infty}(0,X)$
to \eqref{Yeqn} to obtain a solution $Y$ such that
\[
\displaystyle
\exp\left(-\int_\xi^{X}\mu(q)dq\right)Y(\omega)=
    -\delta^2\int_{\omega}^{X} \exp\left(-\int_\eta^\xi\mu(q)dq\right)R(\xi)R^{-1}(\eta)F(\eta)d\eta+O(\delta^3).
\]
Substituting this solution $Y(\omega)$ into \eqref{w24} completes the basis of solutions
of \eqref{evans5} for $0<\delta\ll 1$.

With the above preparations, we are now ready to expand the Evans function in $\delta$.
At $\xi=0$, the resolvent matrix $\mathcal{R}$ of \eqref{evans5} reads
\[
\mathcal{R}(0,\lambda,\delta)=\left(\begin{array}{cc} Id_{\mathbb{C}^3} & \exp\left(\int_0^{X}\mu(q)dq\right)O(\delta^2)\\
 0 & \exp\left(\int_0^{X}\mu(q)dq\right)(1+O(\delta))
 \end{array}\right)
\]
whereas at $\xi=2X$, it reads
 \[
 \mathcal{R}(X,\lambda,\delta)=\left(\begin{array}{cc}\displaystyle R(X,\lambda,\delta) & 0\\
 \displaystyle O(\delta) & 1+O(\delta)\end{array}\right).
\]
Therefore, it follows that
\begin{align*}
&E(\lambda,\xi,\delta)={\rm det}\Big(\mathcal{R}(X,\lambda,\delta)-e^{i\xi X }\mathcal{R}(0,\lambda,\delta)\Big)\\
&\quad=
-(1+O(\delta))e^{i\xi X}\exp\left(\int_0^{X}\mu(q)dq\right)
\left({\rm det}\Big(R(X,\lambda,\delta)-e^{i\xi X}Id_{\mathbb{C}^3}\Big)+O(\delta^3)\right),
\end{align*}
where we have expanded the Evans function with respect to the last column of the determinant to
obtain the final equality.
Recalling that $R(\cdot,\lambda,\delta)=R_{kdv}(\cdot,\lambda)+O(\delta)$, the proposition follows.
\end{proof}

By now considering the equation $E(\lambda,\xi,\delta)=0$ for $0<\delta\ll 1$ and applying an appropriate
implicit function argument, we deduce that for each fixed $\xi\in[-\pi/X,\pi/X)$ the eigenvalues
of Bloch operator $L_\xi$ expand analytically in $\delta$ as $\delta\to 0$.

\begin{corollary}\label{c:ev-analytic}
Let $\xi\in[-\pi/X,\pi/X)$ be fixed and let $\lambda_\delta(\xi)$ be an eigenvalue 
of $L_\xi$ such that $\lim_{\delta\to 0}\lambda_\delta(\xi)=\lambda_0(\xi)$.  Then
for $0<\delta\ll 1$ the eigenvalue $\lambda_\delta(\xi)$ can be expanded as
$$
\displaystyle
\lambda_\delta(\xi)=\lambda_0(\xi)+\delta\lambda_1(\xi)+\delta^2\lambda_2(\xi)+O(\delta^3).
$$
\end{corollary}

\begin{proof}
It is a consequence of the Fenichel-type reduction of \eqref{sppb} conducted above and the regularity of
the flow of the reduced linearized problem that the Evans function is a smooth function
of $(\lambda,\xi,\delta)$ on any compact subset of $\mathbb{C}\times\mathbb{R}\times\mathbb{R}_+^*$.
For $\xi\neq 0$, the eigenvalue $\lambda_0(\xi)$ is an isolated root of $E(\cdot,\xi,0)=E_{kdv}(\cdot,\xi)$
so that, one has $\partial_\lambda E(\lambda_0(\xi),\xi)\neq 0$; see \cite{BD} for more details.
A straightforward application of the implicit function theorem implies that $\lambda_{\delta}(\xi)$ expands as
$\lambda_\delta(\xi)=\lambda_0(\xi)+\delta\lambda_1(\xi)+\delta^2\lambda_2(\xi)+O(\delta^3)$.  A similar
argument holds when $\lambda_0(0)\neq 0$.

Now, let us consider the eigenvalue $\lambda_0(0)=0$ of the linearized KdV equation.
Notice then that by translation invariance and conservation of mass, coming from the conservative
structure of \eqref{kdv-ks}, $\lambda_\delta(0)=0$ is a root of multiplicity two of
$E(\lambda,0,\delta)$ for all $\delta>0$.  Thus, the Evans function at $\xi=0$ can be expressed as
$E(\lambda,0,\delta)=\lambda^2(\tilde{E}_{kdv}(\lambda)+O(\delta))$ for all $|\lambda|$ sufficiently
small and $0<\delta\ll 1$, where here $\tilde{E}_{kdv}(\lambda)=\lambda^{-2}E_{kdv}(\lambda,0)$.
Then $\lambda=0$ is an isolated root of $\tilde{E}_{kdv}$ with $\partial_\lambda\tilde{E}_{kdv}(0)\neq 0$,
so that we can apply the implicit function theorem again and conclude as in the first case.
\end{proof}

\begin{remark}\label{r:bn}
Notice that the expansion of the eigenvalues provided by Corollary \ref{c:ev-analytic} is precisely the one that is
assumed to exist in the work of Bar and Nepomnyashchy in \cite{BN}.  Note, however, that
this expansion is only valid for $0<\delta\ll \xi$ and, moreover, is only
a uniform asymptotic expansion for $|\xi|\geq \eta>0$, where $\eta$ is an arbitrarily small real number.
As a result, all calculations using
an expansion of the form given in Corollary \ref{c:ev-analytic}
are valid only in this restricted regime and, in particular, are not valid in a sufficiently small neighborhood
of the origin in the spectral plane.
\end{remark}

\mathversion{bold}
\subsection{Expansion of eigenvalues as $\delta\to 0$}\label{s:evexpand}
\mathversion{normal}

In the proof of Proposition \ref{p:evansexpand} we obtained an asymptotic expansion for $0<\delta\ll 1$ of the periodic Evans function
for \eqref{sppb} up to $O(\delta^3)$.  However, an explicit expansion of the eigenvalues of such a spectral
problem is 
often
complicated to obtain by
analytic Evans function techniques.
As an alternative, here we fix a Bloch wave number $\xi\in[-\pi/X,\pi/X)$ and search
{\it directly} for an expansion of the $L^\infty(\RM)$ eigenvalues $\lambda_\delta(\xi)$
and eigenfunctions $u=u(\cdot;\delta,\xi)$
in the form
\begin{equation}\label{blochexpand}
\left\{\begin{aligned}
       &\lambda_{\delta}(\xi)=\lambda_0(\xi)+\delta \lambda_1(\xi)+\delta^2 \lambda_2(\xi)+O(\delta^3)\\
       &u(\cdot;\delta,\xi)=u_0(\cdot;\xi)+\delta u_1(\cdot;\xi)+\delta^2 u_2(\cdot;\xi)+O(\delta^3);
       \end{aligned}\right.
\end{equation}
note that such expansions are guaranteed to exist by Corollary \ref{c:ev-analytic} and the Dunford Calculus.
Now, recall that the spectral problem \eqref{sppb} for the operator $L$ can be written as
\begin{equation}\label{skdvks}
\left\{\begin{aligned}
      &u'''+\big((U-c)u\big)'+\delta\big(u''+u''''\big)+\lambda u=0\\
       &u(x+X)=e^{i\xi X}u(x),
      \end{aligned}\right.
\end{equation}
with 
$u\in L^2_{\rm per}([0, X])$.
For $\delta=0$, it is known by the results of \cite{BD}
that the spectrum lies on the imaginary axis and it is parameterized by
\[
\Im\lambda=\pm 8\sqrt{|\eta-\eta_1||\eta-\eta_2||\eta-\eta_3|},\quad \eta\in(-\infty, \eta_1]\cup[\eta_2, \eta_3],
\]
where  $\eta_1=k^2-1, \eta_2=2k^2-1$ and $\eta_3=k^2$ and $k$ is the elliptic modulus associated to
the underlying elliptic function solution $U\big{|}_{\delta=0}$ of the KdV equation for this particular
period $X$.
Moreover, the Bloch wave number $\xi$
can be written as
$$
\displaystyle
\xi=\frac{N\pi}{2K(k)}\pm\frac{\sqrt{|\eta-\eta_1||\eta-\eta_2||\eta-\eta_3|}}{K(k)}\int_0^{K(k)}\frac{dy}{\eta-k^2+{\rm dn}(y,k)}
$$
for some $N\in\mathbb{N}$.

Before beginning our analysis of the perturbation expansion \eqref{blochexpand}, we make some preliminary remarks
concerning the spectrum of the linearized KdV operator.  
Let
\[
\mathcal{L}:=L\big{|}_{\delta=0}=-\partial_x\big(U_0-c_0\big)-\partial_x^3
\]
denote the linearized KdV operator,
considered as a closed densely defined operator on $L^2(\RM)$, and let $\{\mathcal{L}_\xi:\xi\in[-\pi/X,\pi/X)\}$
denote the associated family of Bloch operators defined on 
$L^2_{\rm per}([0,X])$.  
By the results
of \cite{BD}, $\sigma(\mathcal{L})=\RM i$ corresponding to spectral stability of the underlying
cnoidal wave solution $U_0$.  Furthermore, each $\lambda\in\RM i\setminus\{0\}$ is in the spectrum of $\mathcal{L}$
with multiplicity either $1$ or $3$, in the sense that there exists either a unique $\xi\in[-\pi/X,\pi/X)$
such that $\lambda\in\sigma(\mathcal{L}_{\xi})$ (corresponding to multiplicity $1$)
or else there exist three distinct such $\xi$ (corresponding to multiplicity $3$).
Thus, when expanding such eigenvalues for a fixed $\xi$ one is essentially doing simple perturbation theory.
On the other hand, $\lambda=0$ is an eigenvalue of the Bloch operator $\mathcal{L}_0$, corresponding to $\xi=0$,
with algebraic multiplicity three and geometric multiplicity two.  Indeed, one can easily verify
that ${\rm Ker}(\mathcal{L}_0)$ is two dimensional and $1\in{\rm Ker}(\mathcal{L}_0^2)$;
see \cite{BrJ,BrJK}.  Thus, a separate analysis will be necessary when considering the bifurcation
of the neutral modes of $\mathcal{L}_0$ for $0<\delta\ll 1$.

We now begin our perturbation analysis by considering the continuation of a fixed $\lambda_0\in\RM i\setminus\{0\}$.
By above, there exists either one or three distinct Bloch wave numbers $\xi$ such that $\lambda_0\in\sigma(\mathcal{L}_\xi)$.
Let $m(\lambda_0)\in\{1,3\}$ denote the multiplicity of $\lambda_0$, as defined above, and let
$N(\lambda_0)=\{(\xi_j,\hat{u}_{0,j});j=1,\ldots,m(\lambda_0)\}$ denote the set of distinct Bloch wave numbers
$\xi_j$ associated to $\lambda_0$ together with a corresponding
function 
$\hat{u}_{0,j}\in L^2_{\rm per}([0,X])$
in the null-space of the operator $\mathcal{L}_{\xi_j}-\lambda_0$.
We fix such a pair $(\xi_j,\hat{u}_{0,j})\in N(\lambda_0)$ and set $\lambda_0(\xi_j)=\lambda_0$, $u_0(\cdot;\xi_j)=\hat{u}_{0,j}$, and insert
the expansions \eqref{kdvsolnexpand} and \eqref{blochexpand} into \eqref{skdvks}.  Collecting the $O(\delta^0)$ terms we find that $u_0$ must satisfy
\[
\left\{\begin{aligned}
&u_0'''+\big((U_0-c_0)u_0\big)'+\lambda_0 u_0=0\\
&u_0(x+X)=e^{i\xi_j X}u_0(x),
\end{aligned}\right.
\]
which clearly holds by our choice of $(\lambda_0,u_0)$.
Continuing the expansion, identifying the $O(\delta^1)$ terms implies that $u_1(\cdot;\xi_j)$ and $\lambda_1(\xi_j)$
must satisfy 
\begin{equation}\label{o1}
\left\{\begin{aligned}
&u_1'''+\big((U_0-c_0)u_1\big)'+\lambda_0 u_1+\lambda_1 u_0+(U_1 u_0)'+u_0''+u_0''''=0\\
&u_1(x+X)=e^{i\xi_j X}u_1(x),
\end{aligned}\right.
\end{equation}
where here the function $U_1$ is defined as in Proposition \ref{p:kdvsolnexpand}.
To analyze the solvability of \eqref{o1} we
consider the operator $\mathcal{L}_{0,j}[u]=u'''+\big((U_0-c_0)u\big)'$ defined
for all 
$u\in H^3(0,X)$ such that $u(x+X)=e^{i\xi_j X}u(x)$,
and note then that 
the operator $\mathcal{L}_{0,j}+\lambda_0$ is Fredholm of index $0$ on 
$H63(0,X)$.
In particular, we have 
${\rm Range}(\mathcal{L}_{0,j}+\lambda_0)={\rm Ker} \left((\mathcal{L}_{0,j}+\lambda_0)^*\right)^{\bot}$,
where the adjoint operator of $\mathcal{L}_{0,j}+\lambda_0$ is 
given by
\[
(\mathcal{L}_{0,j}+\lambda_0)^*=-\partial_x^3-(U_0-c_0)\partial_x-\lambda_0,
\]
defined here for all 
$u\in H^3(0,X)$ such that $u(x+X)=e^{i\xi_j X}u(x)$.
Notice that if
$\lambda_0\neq 0$ we easily obtain ${\rm Ker}(\mathcal{L}_{0,j}+\lambda_0)^*$
via the following construction: assuming that $\hat{u}_{0,j}\in{\rm Ker}(\mathcal{L}_{0,j}+\lambda_0)$, one can easily
verify that the function 
$\hat{v}_{0,j}(x)=\int_x^{x+X}\hat{u}_{0,j}(s)ds$ is nontrivial, lies in ${\rm Ker}(\mathcal{L}_{0,j}+\lambda_0)^*$,
and satisfies the boundary condition $\hat{v}_{0,j}(x+X)=e^{i\xi_j X}\hat{v}_{0,j}(x)$.
Thus, for any such $\lambda_0$ and associated Bloch wave number $\xi_j\in[-\pi/X,\pi/X)$, we obtain a complete
basis of ${\rm Ker}(\mathcal{L}_{0,j}+\lambda_0)^*$.
In the case $\lambda_0=0$, corresponding to $\xi_j=0$, however,
this construction yields only constant functions:
indeed, one readily finds as a consequence of the conservative structure of the KdV equation \eqref{kdv} that $1\in{\rm Ker} \mathcal{L}_0^{*}$.
However, we also note by translational invariance of \eqref{kdv} that $U_0\in{\rm Ker}\mathcal{L}_0^{*}$.
In this case, we have again found a complete basis of ${\rm Ker}\mathcal{L}_0^*$ since ${\rm Ker}\mathcal{L}_0^*$ is two dimensional.

With the above preparations, we can now obtain an explicit formula for the $O(\delta)$ correction
of the eigenvalue $\lambda_{\delta}$ given in \eqref{blochexpand}.
First, we consider the case where $\lambda_0\neq 0$ and fix a $\xi_j$ such that $\lambda_0\in\sigma(\mathcal{L}_{\xi_j})$.
Then fixing $\hat{u}_{0,j}\in{\rm Ker}(\mathcal{L}_{0,j}+\lambda_0)$ and setting 
$\hat{v}_{0,j}(x)=\int_x^{x+X}\hat{u}_{0,j}(s)ds$
as above, it follows by the Fredholm alternative that
equation \eqref{o1} has a solution provided the compatibility condition
\begin{equation}\label{lambda1eqn}
\left< \lambda_1 \hat{u}_{0,j}+(U_1\hat{u}_{0,j})'+\hat{u}_{0,j}''+\hat{u}_{0,j}'''';\,\hat{v}_{0,j}\right>=0
\end{equation}
is satisfied, where here $\langle\cdot,\cdot\rangle$ denotes the standard (sesquilinear) inner product on 
$L^2_{\rm per}([0,X])$.
We now give an expression for $\lambda_1$ with respect to functions $\hat{v}_{0,j}$.  To this end, note by definition
we have the identity
\[
\hat{v}_{0,j}'(x)=\hat{u}_{0,j}(x+X)-\hat{u}_{0,j}(x)=(e^{i\xi_j X}-1)\hat{u}_{0,j}(x),
\]
from which it follows that
\[
\left<\hat{u}_{0,j}; \hat{v}_{0,j}\right>=\frac{1}{e^{i\xi_j X}-1}\int_0^{X}\hat{v}_{0,j}'\bar{\hat{v}}_{0,j}dx=\frac{e^{-i\xi_j X/2}}{2\sin(\xi_j X/2)}\Im\left(\int_0^{X}\hat{v}_{0,j}'\bar{\hat{v}}_{0,j}dx\right).
\]
Similar computations yield the following identities:
$$
\begin{array}{lll}
\displaystyle
\left<\big(U_1\hat{u}_{0,j}\big)'; \hat{v}_{0,j}\right>=2ie^{-i\xi_j X/2}\sin(\xi_j X/2)\int_0^{X} U_1|\hat{u}_{0,j}|^2dx,\\
\displaystyle
\left< \hat{u}_{0,j}''; \hat{v}_{0,j}\right>=-2e^{-i\xi_j X/2}\sin(\xi_j X/2)\Im\left(\int_0^{X} \hat{u}_{0,j}'\bar{\hat{u}}_{0,j}dx\right),\\
\displaystyle
\left< \hat{u}_0''''; \hat{v}_{0,j}\right>=2e^{-i\xi_j X/2}\sin(\xi_j X/2)\Im\left(\int_0^{X}\hat{u}_{0,j}''\bar{\hat{u}}_{0,j}'dx\right).
\end{array}
$$
\noindent
Taking real and imaginary parts of \eqref{lambda1eqn}, assuming 
$\int_0^{X}\hat{v}_{0,j}'\bar{\hat{v}}_{0,j}dx\neq 0$
we can identify the real part and imaginary part of $\lambda_1$ via the relations
\begin{equation}\label{rl1}
\begin{array}{ll}
\displaystyle
\Im\left(\int_0^{X}\hat{v}_{0,j}'\bar{\hat{v}}_{0,j}dx\right)\Re(\lambda_1)=
    \Im\left(\int_0^{X}\hat{v}_{0,j}''\bar{\hat{v}}_{0,j}'-\hat{v}_{0,j}'''\bar{\hat{v}}_{0,j}''dx\right),\\
\displaystyle
\Im\left(\int_0^{X}\hat{v}_{0,j}'\bar{\hat{v}}_{0,j}dx\right)\Im(\lambda_1)=-\int_0^{X} U_1|\hat{v}_{0,j}'|^2dx.
\end{array}
\end{equation}
Note that the $O(\delta)$ correction $U_1$ of the underlying periodic profile $U$ only contributes, up to $O(\delta)$,
to the imaginary part of $\lambda$. Furthermore, this contribution clearly vanishes by parity.
Indeed, note that 
$|\hat{v}_{0,j}'|^2=\sin^2(\xi_j X/2)(1+|\lambda_0^{-1}U_0'|^2)$
is an even function whereas, by Proposition
\ref{p:kdvsolnexpand}, $U_1$ is an odd function.  As these functions are both $2X$-periodic, assuming
again that 
$\int_0^{X}\hat{v}_{0,j}'\bar{\hat{v}}_{0,j}dx\neq 0$
the integral which defines $\Im(\lambda_1)$ then vanishes,
implying that $\Im(\lambda_1)=0$: note that this is coherent with the computations in \cite{BN}.
As a result, we have obtained an expansion valid up to order $O(\delta^2)$ for any eigenvalue $\lambda$ such that $\lambda|_{\delta=0}\neq 0$
and 
$\int_0^{X}\hat{v}_{0,j}'\bar{\hat{v}}_{0,j}dx\neq 0$.

\begin{remark}
In \cite{BN}, the authors numerically evaluate the expressions in \eqref{rl1} and find, in particular, that
for any fixed $\lambda_0\in\RM i\setminus\{0\}$ that
the condition 
$\int_0^{X}\hat{v}_{0,j}'\bar{\hat{v}}_{0,j}dx\neq 0$
holds for each $j=1,\ldots,m(\lambda_0)$.
Thus, there is no loss of generality in making this assumption above.
\end{remark}

When $\lambda_0=0$, the construction of the expansion is slightly different. Let us first remark that $\lambda_0=0$ is an eigenvalue of the linearized KdV equation which is triply covered but associated to the unique Floquet coefficient $\xi=0$. Indeed the kernel of
the Bloch operator $\mathcal{L}_0$ 
is two dimensional, spanned by the functions $v_1=U'_0$ and
\[
v_2=\partial_M U_0-\frac{\partial_M c_0}{\partial_k c_0}\partial_k U_0\ =\ 1-[\partial_k c_0]^{-1}\partial_k U_0\ .
\]
Furthermore, it is readily checked that $\mathcal{L}_0(1)=-v_1$, hence $\lambda_0=0$ is an eigenvalue of $\mathcal{L}_0$
with algebraic multiplicity three and geometric multiplicity two; see \cite{BrJ,BrJK} for more details.
Similarly, we remark that when $\delta\neq 0$ we have $U'\in{\rm Ker}L_0$, due to the translation invariance of \eqref{kdv-ks},
and that $L_0(1)=-U'$.  As a result, we expect zero to be an eigenvalue of $L_0$ of algebraic multiplicity two for all $\delta\neq 0$.
Our goal now is to determine an asymptotic expansion of the third neutral eigenvalue of the operator $L_0\big{|}_{\delta=0}=\mathcal{L}_0$
for $0<\delta\ll 1$.

To this end, we continue the vector space spanned by $\{v_j\}_{j=1,2}$, defined above, for
$0<\delta\ll 1$.  We begin by recalling Corollary 
\ref{c:ev-analytic} and expanding the corresponding eigenvalues and eigenvectors
of $L_0$ as
\begin{equation}\label{neutralexpand}
\lambda=\delta\lambda_1+\delta^2\lambda_2+O(\delta^3),\quad u=\hat{u}_0+\delta\hat{u}_1+\delta^2\hat{u}_2+O(\delta^3),
\end{equation}
where here we expect generically $\lambda_1\neq 0$.
Substituting these expansions into the spectral problem $L_0 u=\lambda u$, considered here on 
$L^2_{\rm per}([0,X])$,
we find that collecting the $O(\delta^0)$ terms yields $\mathcal{L}_0[\hat{u}_0]=0$.
Thus, for some constants $A_1^0,A_2^0\in\CM$ we can write $\hat{u}_0=A_1^0 v_1+A_2^0 v_2$.  Similarly,
identifying the $O(\delta)$ terms yields the equation
\begin{equation}\label{contl0}
\displaystyle
\mathcal{L}_0[\hat{u}_1]+\lambda_1\hat{u}_0+(U_1\hat{u}_0)'+\hat{u}_0''+\hat{u}_0''''=0.
\end{equation}
\noindent
Recalling that $\mathcal{L}_0$ is Fredholm of index $0$ on 
$H^3_{\rm per}([0,X])$
with ${\rm Ker}\mathcal{L}_0^*={\rm span}\{1,U_0\}$, it follows that
equation \eqref{contl0} has a solution provided the solvability conditions
$$
\begin{array}{ll}
\displaystyle
\langle\lambda_1\hat{u}_0+(U_1\hat{u}_0)'+\hat{u}_0''+\hat{u}_0'''',1\rangle=0\\
\displaystyle
\langle\lambda_1\hat{u}_0+(U_1\hat{u}_0)'+\hat{u}_0''+\hat{u}_0'''',U_0\rangle=0,\\
\end{array}
$$
hold.
More explicitly, using the parity of $\hat{u}_0$ and $U_0$ the above solvability conditions reduce to
\[
\lambda_{1}A_2^0=0,\quad \left(\lambda_{1}\langle v_2,U_0\rangle+\langle (U_1v_2)'+v_2''+v_2'''',U_0\rangle\right)A_2^0=0.
\]
To avoid $\lambda_1=0$, we find $A_2^0=0$ and $\hat{u}_1=A_1^0 (U_1'-\lambda_1)+A_1^1 v_1+A_2^1 v_2$ for some
constants $A_1^1,A_2^1\in\CM$ where now we require $A_1^0\neq 0$.

Next, we consider $O(\delta^2)$ terms in \eqref{neutralexpand} which, upon substitution into the
spectral problem $L_0u=\lambda u$, must satisfy the equation
\begin{equation}\label{delta2ev}
\mathcal{L}_0[\hat{u}_2]+\lambda_1\hat{u}_1+\lambda_2\hat{u}_0+\left(U_1\hat{u}_1+(U_2-c_2)\hat{u}_0\right)'+\hat{u}_1''+\hat{u}_1''''=0
\end{equation}
where $U_1$ and $U_2$ represent, respectively, the $O(\delta)$ and $O(\delta^2)$ corrections of the underlying wave profile $U$:
see Proposition \ref{p:kdvsolnexpand}.
Using the representations of $\hat{u}_0$ and $\hat{u}_1$ determined above, equation \eqref{delta2ev} can be
written as
{\setlength\arraycolsep{1pt}
\begin{eqnarray}
\displaystyle
\mathcal{L}_0[\hat{u}_2]&+&A_1^0\left(-\lambda_1^2+\lambda_2U_0'+(U_1U_1'+(U_2-c_2)U_0')'+U_1'''+U_1'''''\right)\nonumber\\
\displaystyle
&+&A_2^1\Big(\lambda_1 v_2+(U_1v_2)'+v_2''+v_2''''\Big)+A_1^1\Big(\lambda_1 v_1+(U_1v_1)'+v_1''+v_1''''\Big)=0\nonumber
\end{eqnarray}}
\noindent
or, more compactly, as 
\begin{equation}\label{equ2h}
\mathcal{L}_0[\hat{u}_2+A_1^0(\lambda_2+U_2')+A_1^1(\lambda_1+U_1')]-\lambda_1^2A_1^0+(\lambda_1 v_2+(U_1v_2)'+v_2''+v_2'''')A_2^1=0.
\end{equation}
\noindent
Using the Fredholm alternative again, we find that equation \eqref{equ2h} has a solution provided the solvability conditions
$$
\displaystyle
-\lambda_1^2A_1^0+\lambda_1 A_2^1=0,\quad \left(\lambda_1\langle v_2,U_0\rangle+\langle (U_1v_2)'+v_2''+v_2''''; U_0\rangle\right)A_2^1=0
$$
are satisfied.  In particular, notice that these solvability conditions provide no requirement for the constant $A_1^1$.
Simplifying, we have thus obtained the following dispersion relation
\begin{equation}\label{neutraldisp}
\lambda_1^2\left(\lambda_1\langle v_2,U_0\rangle+\langle (U_1v_2)'+v_2''+v_2''''; U_0\rangle\right)=0
\end{equation}
defining the $O(\delta)$ corrector $\lambda_1$ in \eqref{neutralexpand}.
As a result $\lambda_1=0$ is a solution and it is of multiplicity $2$ corresponding to
the Jordan block of height $2$ (up to order $O(\delta^2)$).
In this case, one has necessarily $A_2^1=0$ and a corresponding eigenfunction expands as $u=A_1^0 U'+O(\delta^2)$.
Provided $\langle v_2,U_0\rangle\neq 0$, the third solution of the dispersion relation \eqref{neutraldisp}
is given by
\begin{equation}\label{neutralbreak}
\lambda_1=-\frac{\langle (U_1v_2)'+v_2''+v_2''''; U_0\rangle}{\langle v_2,U_0\rangle}\in\mathbb{R}.
\end{equation}
In this case, $A_2^1=\lambda_1 A_1^0$ and an associated eigenvector expands as
$\displaystyle u=U_0'+\delta\big(U_1'-\lambda_0+\lambda_0 v_2\big)+O(\delta^2)$.

As a consequence of the above analysis, which is a rigorous version of the formal analysis provided in \cite{BN},
for a fixed $\xi\in[-\pi/X,\pi/X)$ we have explicit expressions, given by \eqref{rl1} and \eqref{neutralbreak},
for the eigenvalues of the Bloch operator $L_\xi$
as they bifurcate from the eigenvalues of the associated Bloch operator $\mathcal{L}_\xi$ for the KdV equation.
In \cite{BN}, the authors numerically evaluate these expressions for each fixed $\xi$ using standard elliptic function calculations:
the details of these calculations are provided in Appendix \ref{s:bn}.  In particular, the authors of \cite{BN} find
that for each fixed $\xi$ the $O(\delta)$ correctors $\lambda_1$ are strictly negative
for wave trains of \eqref{kdv-ks} having periods lying in the interval $[8.49,26.17]$, indicative of spectral stability
of the associated wave trains.  However, as described in Remark \ref{r:bn} this analysis is only valid for $0<\delta\ll \xi$
and, furthermore, the expansions assumed above are only uniform for $\xi$ bounded away from zero.  As a result,
the previous analysis is not sufficient to conclude spectral stability of a given periodic traveling wave $U$
of \eqref{kdv-ks} for any fixed $\delta>0$.  To make such a conclusion, delicate analysis in a neighborhood of the origin
in the spectral plane is needed: this is the objective of the next section.

Finally, we conclude this section by connecting the above analysis with the nonlinear stability 
theory developed in \cite{BJNRZ1}.
As described in the introduction, nonlinear 
stability of the underlying profile $U$
under the nonlinear flow induced by \eqref{kdv-ks} follows by the structural and spectral 
hypotheses (H1)-(H2) and (D1)-(D3).
Yet as a consequence of the above analysis, assumptions (H1) and (D3) immediately follow from verifying
that the $O(\delta)$ corrector $\lambda_1$ given in \eqref{neutralbreak} is non-zero.  This observation is
recorded in the following corollary.

\begin{corollary}\label{c:h1d3}
Assume that the number $\lambda_1$ defined in \eqref{neutralbreak} is non-zero.
Then $\lambda=0$ is a non-semi simple eigenvalue of the Bloch operator $L_0$: it is of algebraic
multiplicity two and geometric multiplicity one with ${\rm Ker}L_{\xi=0}={\rm span}\{U'\}$ and
$1\in{\rm Ker} L_{\xi=0}^2\setminus{\rm Ker} L_{\xi=0}$.  Furthermore, the return map 
$H:\R^6\to \R^3$, $(X,b,c,q)\mapsto (u,u',u'')(X,b,c,q)-b$ (where $(u,u',u'')(.,b,c,q)$ is solution of
\[
\delta(u'+u''')+u''+\frac{u^2}{2}-cu=q,\quad u(0,b,c,q)=b)
\]
is full rank at $(\bar X,\bar b,\bar c,\bar q)$.
\end{corollary}

\begin{proof}
The non-semi simplicity of the zero-eigenspace of $L_0$ follows by the above considerations.
The fact that the return map is full rank is a consequence of the fact that $0$ is of
algebraic multiplicity $2$: see 
\cite{NR2,JNRZ1} for more details.
\end{proof}

\section{\label{sec4}Spectrum at the origin and modulation equations}

As described above, the computations carried out in Section \ref{s:evexpand}, which justify the formal approach in \cite{BN}, are only valid for $|\xi|\geq\eta>0$ and $\delta\to 0$, where $\eta>0$ is arbitrary.  In particular, it can be used to provide a first estimate of stability boundaries
as any instabilities undetected would correspond to long-wavelength perturbations.
Indeed it is verified numerically in \cite{BN}, based on an expansion of eigenvalues similar to the one carried out in the previous section,
that $\sigma(L_\xi)\subset\{\lambda\in\RM i:{\rm Re}(\lambda)<0\}$ holds for all $|\xi|\geq\eta>0$, $\eta>0$ sufficiently small,
for $X\in[L_1, L_2]$
with $L_1\approx 8.49$ and $L_2\approx 26.17$: see \cite{BN}
or Appendix \ref{s:bn}
for more details. 
However, one can not conclude directly to spectral stability 
since 
the above analysis does not rule out the presence of unstable spectrum in a sufficiently small neighborhood of the origin.
Nevertheless, we point out that, somewhat surprisingly, these bounds found in \cite{BN}
are approximately those found through a direct numerical analysis of the spectral problem conducted recently in \cite{BJNRZ1}.
In this section, we complete the stability analysis initiated in the previous section by studying stability
of low Bloch numbers $|\xi|\leq\eta$ for $\eta$ sufficiently small.  In the process of verifying the spectral
stability hypothesis (D1), required for the nonlinear stability result of \cite{BJNRZ1} to apply,
we also prove rigorously the hypotheses (H2) (``hyperbolicity'') and (D2) (``dissipativity'').  As a result, in conjunction with Corollary \ref{c:h1d3} and the numerical results of \cite{BN}, our results indicate\footnote{Our results do not prove the existence of such solutions since it
still relies on the numerical results of \cite{BN}.  As previously indicated, making these numerics rigorous via numerical proof would be an interesting direction for future investigation.}
the existence of nonlinearly stable periodic traveling wave solutions of \eqref{kdv-ks} in the sense of that defined in \cite{BJNRZ1}.

\begin{figure}[htbp]\label{fig1}
\begin{center}
\input{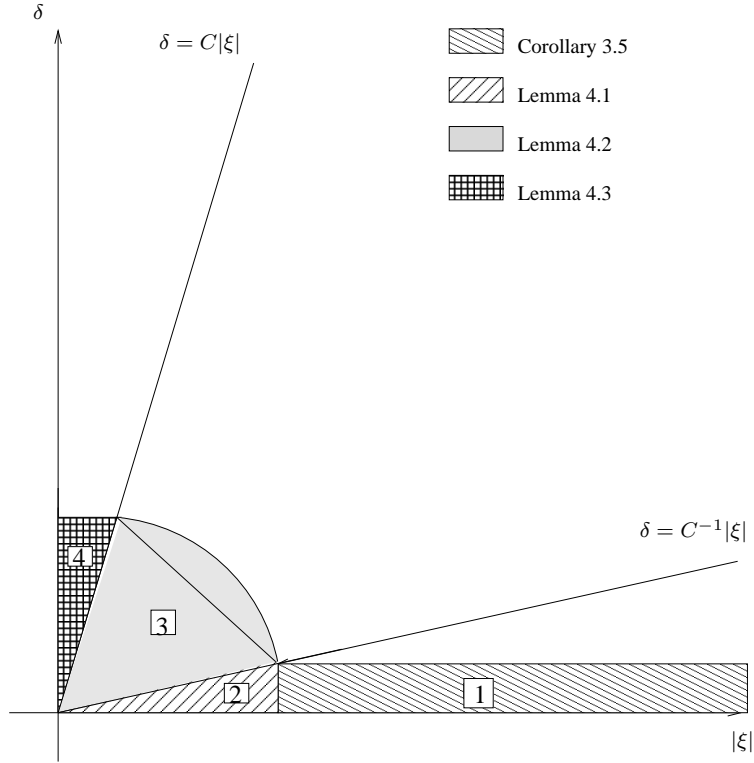}
\caption{ Domains in the $(|\xi|,\delta)$ plane where it is proved that there are no unstable eigenvalues. The sector is described by $C^{-1}|\xi|\leq\delta\leq C|\xi|$ with $C\gg 1$. The zone $1$ corresponds to the domain of validity of \cite{BN}. The zone $2$ corresponds to Lemma \ref{lem_kdv}: it is an extension of the analysis of \cite{BN} and provides subcharacteristic conditions (S1), (S2), (S3) of stability. In zone $3$, Lemma \ref{lem_int} shows that no eigenvalue can cross the imaginary axis, thus no unstable eigenvalue appears there.  In zone $4$, Lemma \ref{lem_wh} proves that there are no unstable eigenvalues if subcharacteristic conditions (S1), (S2), (S3) are satisfied. One concludes to spectral stability for $\delta\leq\delta_0$ for $\delta_0$ sufficiently small.}
\end{center}
\end{figure}

\subsection{Spectral analysis through Evans function computations}\label{s:specevans}

We begin our study of the spectrum of the linearized operator $L$ in a neighborhood of the origin by analyzing
the periodic Evans function $E(\lambda,\xi,\delta)$ for $|(\lambda,\xi,\delta)|_{\CM\times\RM\times\RM}\ll 1$.
Recall from Proposition \ref{p:evansexpand} that, after a suitable renormalization, the Evans function expands as
\[
E(\lambda,\xi,\delta)=E_{kdv}(\lambda,\xi)+\delta E_1(\lambda,\xi)+O\left(\delta^2(\xi^2+\lambda^2)\right)
\]
for sufficiently small $\delta>0$.  
Note that due to the algebraic multiplicity of the root $\lambda=0$ of $E_{kdv}(\cdot,0)=0$,  
the principal part of $E_{kdv}$ in its Taylor expansion with respect to $(\lambda,\xi)$ 
is an homogeneous polynomial of degree 3, whereas $E_1$ has a homogeneous polynomial of degree 2 as 
a principal part in its Taylor expansion about $(0,0)$.

Restricting to $|(\lambda,\xi)|_{\CM\times\RM}\ll 1$, the
Weierstrass Preparation Theorem yields an expansion of the form
\[
E(\lambda,\xi,\delta)=\Gamma(\lambda,\xi)(\lambda-i\alpha_1(\xi)\xi)(\lambda-i\alpha_2(\xi)\xi)(\lambda-i\alpha_3(\xi)\xi)+\delta E_1(\lambda,\xi)+O\left(\delta^2(\lambda^2+\xi^2)\right),
\]
where $\Gamma$ is an analytic function such that $\Gamma(0,0)\neq 0$ and the numbers $i\alpha_j(\xi)\xi$
are the roots of the associated Evans function $E_{kdv}(\cdot,\xi)$ for the linearized KdV equation.
Using primarily the above asymptotic expansion of the periodic Evans function,
the description of the spectrum of $L$ near the origin is done in three steps.
First, we show that the computations carried out in Section \ref{s:evexpand} remain
valid in a sector of the form $0<\delta\leq\epsilon_0|\xi|$ and $|\xi|\leq\eta_0$, for some $0<\epsilon_0,\eta_0\ll 1$.
Next, we show that eigenvalues of $L_\xi$, considered here as a family of operators on 
$L^2_{\rm per}([0,X])$
indexed by $\xi$, can not cross the imaginary axis, except at $\xi=0$, in a sector of the form
$C^{-1}|\xi|\leq \delta\leq C|\xi|$ and $|\xi|\leq\eta_C$, where here $\eta_C>0$ is small and $C\geq 1$ is arbitrary.  Finally, we
consider a sector $|\xi|\leq\epsilon_0\delta$ and show that $\Re(\lambda(\xi,\delta))\leq-\theta(\delta)\xi^2$
if some subcharacteristic conditions are met.

For $(\lambda,\xi)$ sufficiently small, we divide the Evans function by $\Gamma(\lambda,\xi)\neq 0$ and
expand it with respect to $\lambda,\xi,\delta$ as
\begin{equation}\label{crit-evansexpand}
E(\lambda,\xi,\delta)=\prod_{j=1}^3(\lambda-i\alpha_j(\xi)\xi)+\gamma\delta\prod_{k=1}^2(\lambda-i\beta_k^0\xi)+O\left(\delta^2(\lambda^2+\xi^2)+\delta(\lambda^3+\xi^3)\right).
\end{equation}
where $\gamma,a,b\in\mathbb{R}$ are constants, the $\beta_k$ are real or complex conjugate constants, and $\alpha_j(\xi)\in\mathbb{R}$.
Notice that since the spectral curves for the linearized KdV equation obey the symmetry $\lambda(-\xi)=\bar{\lambda}(\xi)$,
it follows that the $\alpha_j$ are even functions of $\xi$.  Furthermore, letting $\xi\to 0$ in $\alpha_j(\xi)$ one
obtains $\alpha_j(0)=\alpha_j^0, j=1,2,3$, where the $\alpha_j^0$ are the eigenvalues of the
Whitham modulation system for Korteweg-de Vries equation: see \cite{BrJ,BrJK,JZB,JZ1} for details.  
Concerning the $\alpha_j^0$, we note that there are well-founded numerical studies demonstrating that these eigenvalues
are distinct for all the KdV cnoidal wave trains, i.e. that the Whitham modulation system for
the KdV equation is strictly hyperbolic at all such solutions; see Section 5.1 of \cite{BrJK} for instance.  
In the present case 
though\footnote{Note that, 
since all quantities are analytic, failure of strict hyperbolicity can occur in any case, if it does, only at isolated periods.},
we find it more appropriate to note
that the numerical results
of Figure \ref{fig2} in Section \ref{s:num} below clearly shows strict
hyperbolicity for the limiting class of KdV cnoidal waves considered here, i.e. those
wave trains of \eqref{kdv} approachable as solutions
of the KdV-KS equation \eqref{kdv-ks} as $\delta\to 0$ described in Proposition \ref{p:kdvsolnexpand}.
Throughout this section, we assume that the $\alpha_j^0$ are distinct and, 
without loss of generality, obey the ordering
\begin{equation}\label{vap_kdv}
\displaystyle
\alpha^0_1<\alpha_2^0<\alpha_3^0.
\end{equation}
As a first step, we prove that the expansions carried out in the previous section,
that are valid for $|\xi|>\eta$ and $|\lambda|>\eta$ for any $\eta>0$ and $\delta\to 0$
extend to a small sector $0<\delta\leq\epsilon|\xi|$ and $|\xi|<\eta$ for 
$\eta,\epsilon>0$ sufficiently small.
This is the content of the following lemma.

\begin{lemma}\label{lem_kdv}
There exist constants $\epsilon_0,\eta_0>0$ and $M_0>0$ so that for all $0\leq \delta\leq\epsilon_0|\xi|$ and $|\xi|\leq\eta_0$, there are
only three roots $\{\lambda_k(\xi,\delta)\}_{k=1,2,3}$ of
$E(\lambda,\xi,\delta)$ with $|\lambda|\leq M_0$.
Moreover, these roots are smooth functions of $\xi$ and $\delta/\xi$ and expand as
$$
\begin{array}{ll}
\displaystyle
\lambda_k(\xi,\delta)=i\alpha_k(\xi)\xi-\gamma\delta\frac{\prod_{j=1}^2(\alpha_k(\xi)-\beta_j^0)}{\prod_{j\neq k}(\alpha_k(\xi)-\alpha_j(\xi))}+O(\delta\xi),\\
\displaystyle
\Re(\lambda_k(\xi,\delta))=\delta A_k+O(\delta\xi^2),\quad A_k=-\gamma\frac{\prod_{j=1}^2(\alpha_k^0-\beta_j^0)}{\prod_{j\neq k}(\alpha_k^0-\alpha_j^0)}
\end{array}
$$
for $|(\xi,\delta/\xi)|\ll 1$.
\noindent
One has $A_k<0, k=1,2,3$ if and only if the following conditions are satisfied:
\begin{itemize}
\item[{\bf (S1)}] $\beta_1^0,\beta_2^0\in\mathbb{R}$ and $\beta_1^0\neq\beta_2^0$;
\item[{\bf (S2)}] $\alpha_1^0<\beta_1^0<\alpha_2^0<\beta_2^0<\alpha_3^0$ (once we have fixed $\beta_1^0<\beta_2^0$).
\item[{\bf (S3)}] $\gamma>0$;
\end{itemize}
Then, up to a restriction on $\eta_0,\epsilon_0$, $\Re(\lambda_k(\xi,\delta))<0$ for all $0<\delta\leq\epsilon_0|\xi|$ and $|\xi|\leq\eta_0$.
\end{lemma}

\begin{remark} In what follows, the conditions $(S1),(S2),(S3)$ will be referred to as
``the subcharacteristic conditions'': this terminology will be justified in Section \ref{whitham}
below.  There, we will see that the $\beta_j^0$ are the characteristics of the first order averaged Whitham
modulation equations for \eqref{kdv-ks}.  Hence, condition (S1) above simply states that the Whitham modulation system for \eqref{kdv-ks}, derived for fixed $\delta>0$, about the underlying wave is strictly hyperbolic.  Note that hyperbolicity of this system, corresponding to the requirement that $\beta_j^0\in\RM$, is a well-known necessary condition for spectral stability to weak large-scale perturbations; see \cite{Se,NR2}.  We note furthermore that the condition $(S1)$ is equivalent to the spectral assumption $(H2)$ necessary to invoke the nonlinear stability theory of \cite{BJNRZ1}.
\end{remark}

\begin{proof}
As described above, for $(|\lambda|,|\xi|,\delta)$ sufficiently small the equation $E(\lambda,\xi,\delta)=0$
expands as
\begin{equation}\label{ev1}
\displaystyle
\prod_{j=1}^3(\lambda-i\alpha_j(\xi)\xi)+\gamma\delta\prod_{k=1}^2(\lambda-i\beta_k^0\xi)+O\left(\delta^2(\lambda^2+\xi^2)+\delta(\lambda^3+\xi^3)\right)=0,
\end{equation}
\noindent
where $O\left(\delta^2(\lambda^2+\xi^2)+\delta(\lambda^3+\xi^3)\right)$ stands for
an analytic function of $\lambda,\delta,\xi$ of the given order.
Now, dividing \eqref{ev1} by $\xi^3$ and setting $\delta=\bar{\delta}\xi, \lambda=\bar\lambda\xi$
yields the equation
\begin{equation}\label{ev1rescale}
\prod_{j=1}^3(\bar\lambda-i\alpha_j(\xi))+\gamma\bar\delta\prod_{k=1}^2(\bar\lambda- i\beta_k^0)
   =O\left(\xi\bar\delta^2(1+\bar\lambda^3)+\bar\delta\xi(1+\bar\delta^2)\right).
\end{equation}
By comparing polynomial growth in $\bar\lambda$, it follows that there exists
constants $\eta_1,M>0$, such that if $|\bar\delta|+|\xi|<\eta_1$ then $|\bar\lambda|\leq M<\infty$.

Now, letting $\bar\delta\to 0$ in \eqref{ev1rescale} one finds that necessarily
$$
\displaystyle
\prod_{j=1}^3(\bar\lambda-i\alpha_j(\xi))=0.
$$
Since $\alpha_1^0<\alpha_2^0<\alpha_3^0$, the continuity of the $\alpha_j(\cdot)$ implies the existence
of an $\eta_2>0$ such that $\alpha_1(\xi)<\alpha_2(\xi)<\alpha_3(\xi)$ for all $|\xi|\leq\eta_2$.
Thus, applying the implicit function theorem to \eqref{ev1rescale} in a neighborhood of each $\alpha_j(\xi)$
it follows that there exists three roots $\{\bar\lambda_j(\xi,\bar\delta)\}_{j=1,2,3}$ of \eqref{ev1rescale},
defined for $\xi,\bar\delta$ sufficiently small, which
are smooth functions of $\xi$ and $\bar{\delta}$ and can be expanded as
\begin{equation}\label{rescale-critexpansion}
\bar\lambda_j(\xi,\bar\delta)=i\alpha_j(\xi)
    -\gamma\bar{\delta}\frac{(\alpha_j(\xi)-\beta_1^0)(\alpha_j(\xi)-\beta_2^0)}{\prod_{k\neq j}(\alpha_j(\xi)-\alpha_k(\xi))}
        +O(\bar\delta\xi):
\end{equation}
notice here we have used the fact that the $\alpha_j$ are even functions of $\xi$.
Returning to the original variables via $\lambda_j(\xi,\delta)=\xi\bar{\lambda}_j(\xi,\bar{\delta})$,
we obtain the desired regularity and expansions for the critical eigenvalues $\{\lambda_j\}$.

Next, we compute the
real parts
of the critical eigenvalues $\lambda_j(\xi,\delta)=\xi\bar{\lambda}_j(\bar\delta,\xi)$.
Recalling that the constants $\beta_1^0,\beta_2^0$ are either real or complex conjugate, as well as the fact
that the functions $\Re(\lambda_j(\xi,\delta)),\alpha_j(\xi)$ are even function of $\xi$, it follows by
\eqref{rescale-critexpansion} that
\begin{equation}\label{re_lambda}
\displaystyle
\Re(\lambda_j(\xi,\delta))=-\gamma\delta\frac{(\alpha_j^0-\beta_1^0)(\alpha_j^0-\beta_2^0)}{\prod_{k\neq j}(\alpha_j^0-\alpha_k^0)}+O(\delta\xi^2).
\end{equation}
\noindent
Hence, for $|\xi|$ and $\delta/|\xi|$ sufficiently small, the sign of
$\Re(\lambda_j(\xi,\delta)), j=1,2,3$ is determined by the sign of the real number $A_j$ defined as
\[
A_j=-\gamma\frac{(\alpha_j^0-\beta_1^0)(\alpha_j^0-\beta_2^0)}{\prod_{k\neq j}(\alpha_j^0-\alpha_k^0)},\quad j=1,2,3.
\]
One now needs to verify that $A_j<0$ for $j=1,2,3$.

For the moment, let us assume that $A_j<0$ for $j=1,2,3$ and demonstrate that this implies
the conditions (S1), (S2), and (S3) are satisfied.
First, we suppose that $\beta_1^0,\beta_2^0$ are complex conjugates.  In this case our assumption
on the $A_j$ implies that $\gamma\neq 0$ and $(\alpha_j^0-\beta_1^0)(\alpha_j^0-\beta_2^0)=|\alpha_j^0-\beta_1^0|^2>0$ for each $j$, hence
\[
\sign(A_j)=-\sign(\gamma)\sign(\prod_{k\neq j}(\alpha_j^0-\alpha_k^0)).
\]
Since $\alpha_1^0<\alpha_2^0<\alpha_3^0$, it follows that, contrary to our hypothesis, the $A_j$ can
not have all the same sign.
Thus, it must be the case that the $\beta_j^0$ are real and distinct, verifying condition (S1).
Taking without loss of generality $\beta_1^0<\beta_2^0$, it is now an easy computation
to show that the signs of $(A_j)_{j=1,2,3}$ are the same if and only if the condition (S2) is satisfied.
In this case, one has $\sign(A_j)=-\sign(\gamma)$ for $j=1,2,3$ hence, in order to have
$A_j<0$ for each $j$, it follows that condition (S3) must hold.  This verifies that conditions (S1), (S2),
and (S3) hold provided $A_j<0$ 
for each $j=1,2,3$. Conversely, it is a elementary computation to show that conditions
(S1), (S2), and (S3) imply $A_j<0$ for $j=1,2,3$. This completes the proof of the lemma.
\end{proof}


In \cite{BN}, the authors computed numerically a stability index that we denote  $\textrm{Ind}(X)$ here. It is defined as follows:
assume that $\lambda_n(\xi,\delta), n\in\mathbb{N}$ is one of the solutions
 of $E(\lambda,\xi,\delta)=0$ expanding as $\lambda_n(\xi,\delta)=\lambda_n^0(\xi)+\delta\lambda_n^1(\xi)+O(\delta^2)$, with $\lambda_n^0(\xi)$
an eigenvalue of the linearized KdV equation about a $X$-periodic traveling wave. Then ${\rm Ind}(X)$ is defined as
$$
\displaystyle
{\rm Ind}(X)=\max_{\xi\in[-\pi/X,\pi/X), n\in\mathbb{N}}(\Re\lambda_n^1(\xi))),
$$
\noindent
and the authors in \cite{BN} conclude spectral stability if ${\rm Ind}(X)<0$ and instability if ${\rm Ind}(X)>0$. 
However, the previous lemma shows that for 
$(\lambda,\xi)$ sufficiently small the expansion assumed in the definition of $\lambda_n(\xi,\delta)$ is
apriori only valid in a sector of the form 
$0<\delta<\epsilon_0|\xi|$ and $\delta<\eta_0|\xi|$, hence one can not conclude spectral stability
from such an expansion,
as is done in \cite{BN}, when $\textrm{Ind}(X)<0$.  However, it is easy to see that the $\{A_j\}$ from Lemma~\ref{lem_kdv} satisfy $A_j\leq {\rm Ind}(X)$ for each $j=1,2,3$.  Hence, for the ``near-KdV" wave trains of \eqref{kdv-ks} satisfying ${\rm Ind}(X)<0$, one deduces that the subcharacteristic conditions (S1),(S2), and (S3) are satisfied so that, in particular, the spectral stability in Zone 1 in Figure \ref{fig1} extends to Zone 2 for such waves. In what follows, we prove the requirement that $\textrm{Ind}(X)<0$ is indeed \textit{sufficient} to conclude to spectral stability of such a periodic wave.

To this end, we must now deal with the range of parameters $\delta>\epsilon_0|\xi|$.
In what follows, we will assume that the conditions (S1)-(S3) are satisfied.  
This last part is done in two steps.  First, we demonstrate that under this assumption no eigenvalues in a sector
of the form $C^{-1}\xi\leq\delta\leq C\xi$ for $C\gg 1$ sufficiently large and $(|\lambda|,|\xi|)$ sufficiently
small can cross the imaginary axis except at the degenerate boundary point $(\xi,\delta)=(0,0)$.  
Then we show that in a sector of the form 
$|\xi|\leq \epsilon\delta$ and $(\epsilon,\delta)$ sufficiently small, one again
has spectral stability in a neighborhood of the origin provided the subcharacteristic conditions (S1)-(S3) are satisfied.

\begin{lemma}\label{lem_int}
Assume the conditions (S1), (S2), and (S3) of Lemma \ref{lem_kdv} hold. Then for all $C>1$,
there exists a constant $\eta>0$ such that
the roots of the function $\lambda\mapsto E(\lambda,\xi,\delta)$ with $|\lambda|\leq\eta$
can not cross the imaginary axis for any $|\xi|\leq \eta$ and $\frac{|\xi|}{C}\leq\delta\leq C|\xi|$,
except at $\lambda=\xi=0$.
\end{lemma}

\begin{proof}
Assume that a root $\lambda=\lambda(\xi,\delta)$  of the equation $E(\lambda,\xi,\delta)=0$, with $|\lambda|$ sufficiently small so that the expansion \eqref{crit-evansexpand} is valid, 
crosses the
imaginary axis as the parameters $\xi$ and $\delta$ are varied in the region
$|\xi|\leq \eta$ and $\frac{|\xi|}{C}\leq\delta\leq C|\xi|$.
Then there exists $\tau=\tau(\delta,\xi)\in\mathbb{R}$
such that $E(i\tau,\xi,\delta)=0$ from which, by identifying the real and imaginary part of this equation
and using \eqref{crit-evansexpand}, one finds
$$
\begin{array}{ll}
\displaystyle
\prod_{j=1}^3(\tau-\alpha_j(\xi)\xi)=O\left(\delta(\tau^3+\xi^3)+\delta^2(\tau^2+\xi^2)\right),\\
\displaystyle
\prod_{j=1}^2(\tau-\beta^0_j\xi)=O\left((\tau^3+\xi^3)+\delta(\tau^2+\xi^2)\right).\\
\end{array}
$$
Dividing the first of these equations by $\xi^3$ and the second one by $\xi^2$,
and setting $\tau=\bar{\tau}\xi$ and using the estimate $|\xi|\leq C\delta$, one finds
\[
\prod_{k=1}^3(\bar{\tau}-\alpha_k(\xi))=O(\delta(\bar\tau^3+1)),\quad \prod_{j=1}^2(\bar\tau-\beta_j^0)=O(\delta(\bar{\tau}^3+1)).
\]
By comparing polynomial growth in $\bar\tau$ in the first equation it follows that $|\bar\tau|\leq M$ for some constant $M$.
Furthermore, since the functions $\alpha_j$ are even in $\xi$ and $|\xi|\leq C\delta$, we have
\[
\prod_{k=1}^3(\bar\tau-\alpha_k^0)=O(\delta),\quad \prod_{k=1}^2(\bar\tau-\beta^k_0)=O(\delta).
\]
Taking $\delta>0$ sufficiently small, one finds a contradiction with condition (S2).
\end{proof}

As a consequence of Lemma \ref{lem_kdv} and Lemma \ref{lem_int}, it follows that
if the conditions (S1), (S2), and (S3) are satisfied, for any 
$C$ sufficiently large, there exists an $\eta>0$
such that if 
$\delta\leq C|\xi|$ and $|\xi|\leq\eta$
any solution $\lambda$ of $E(\lambda,\xi,\delta)=0$ such that $|\lambda|\leq\eta$ and 
$(\lambda,\xi)\neq(0,0)$ must satisfy $\Re(\lambda)<0$. Next, we verify that, under the same conditions, for $\epsilon>0$ small enough, in the sector $\delta\leq\epsilon|\xi|$ and $0\leq|\xi|\leq\epsilon$, all eigenvalues $\lambda$, solutions of $E(\lambda,\xi,\delta)=0$ such that $|\lambda|\leq\epsilon$ and 
$(\lambda,\xi)\neq(0,0)$, have a negative real part.

\begin{lemma}\label{lem_wh}
There exist constants $\epsilon_0,\eta_0>0$ and $M_0>0$ such that for any $0\leq |\xi|\leq\epsilon_0\delta$ and $0<\delta\leq\eta_0$,
there are only three roots $\{\lambda_j(\xi,\delta)\}_{j=1,2,3}$ of $E(\lambda,\xi,\delta)$ with $|\lambda|\leq M_0$.
Moreover, assuming (S1) and $\gamma\neq 0$,
these roots are smooth functions of $\delta$ and $\xi/\delta$ and their real part
expand as $\Re(\lambda_3(\xi,\delta))=-\gamma\delta+o(\delta)$ and 
$$
\displaystyle
\Re(\lambda_1(\xi,\delta))=\frac{\xi^2}{\gamma\delta}\frac{\prod_{j=1}^3(\beta_1^0-\alpha_j^0)}{(\beta_1^0-\beta_2^0)}+O(\xi^2),\quad
\Re(\lambda_2(\xi,\delta))=\frac{\xi^2}{\gamma\delta}\frac{\prod_{j=1}^3(\beta_2^0-\alpha_j^0)}{(\beta_2^0-\beta_1^0)}+O(\xi^2).
$$
In particular, if conditions (S1), (S2), and (S3) are satisfied,
then, up to possibly choosing $(\epsilon_0,\eta_0)$ smaller than above,
there exists $\theta>0$ such that for $i=1,2$
$$
\displaystyle
\Re(\lambda_i(\xi,\delta))\leq -\frac{\theta}{\delta}\xi^2,\quad  0\leq|\xi|\leq\epsilon_0\delta,\quad \delta<\eta_0.
$$
\end{lemma}

\begin{remark}
In particular, notice that Lemma \ref{lem_wh} validates, under the hypothesis that (S1)-(S3) hold,
the ``dissipativity" condition (D2) of the nonlinear stability theory of \cite{BJNRZ1}.  Furthermore,
it follows that the $O(\delta)$ corrector $\lambda_1$ in \eqref{neutralbreak} is non-zero provided
that $\gamma\neq 0$.  In particular, by Corollary \ref{c:h1d3} it follows that the
conditions (H1) and (D3)  hold provided $\gamma\neq 0$.
\end{remark}

\begin{proof}
Recall from \eqref{crit-evansexpand} that for $(|\lambda|,|\xi|,\delta)$ sufficiently small the equation $E(\lambda,\xi,\delta)=0$
expands as 
\begin{equation}\label{ev2}
\displaystyle
\prod_{j=1}^3(\lambda-i\alpha_j(\xi)\xi)+\gamma\delta\prod_{k=1}^2(\lambda-i\beta_k^0\xi)=O\left(\delta(\lambda^3+\xi^3)+\delta^2(\lambda^2+\xi^2)\right).
\end{equation}
Dividing this equation by $\delta^3$ and setting $\lambda=\bar\lambda\delta$, $\xi=\bar\xi\delta$ yields the equation
\begin{equation}\label{ev2.5}
\prod_{j=1}^3(\bar \lambda-i\alpha_j(\delta\bar\xi)\bar\xi)+\gamma\delta\prod_{k=1}^2(\bar\lambda-i\beta_k^0\bar\xi)=O\left(\delta(\bar\lambda^3+\bar\xi^3)+\delta(\bar\lambda^2+\bar\xi^2)\right),
\end{equation}
which can be rewritten as
\begin{equation}\label{ev3}
(1+O(\delta))\bar\lambda^3+(\gamma+O(\delta+\bar\xi))\bar\lambda^2=O(\bar\xi^2+\bar\lambda\bar\xi)+O(\delta\bar\xi^2)
\end{equation}
Hence, for $|\bar\xi|,\delta$ sufficiently small, one shows by
comparing polynomial growth on $\bar\lambda$ that $|\bar\lambda|\leq M$ for some $M>0$.
Now, letting $\bar\xi,\delta\to 0$ in \eqref{ev3} we have that
\[
\bar\lambda^2(\bar\lambda+\gamma)=0,
\]
hence $\bar\lambda=-\gamma$ is an isolated root of \eqref{ev2.5} when $(\bar\xi,\delta)=(0,0)$.
Applying the implicit function theorem to \eqref{ev2.5} in a neighborhood of $(\bar\lambda,\bar\xi,\delta)=(-\gamma,0,0)$
implies the existence of a root $\lambda_3(\xi,\delta)$ of \eqref{ev2}, smooth in $\bar\xi$ and $\delta$ and
defined for $\bar\xi,\delta$ sufficiently small, which expands as
\[
\lambda_3(\xi,\delta)=-\gamma\delta+o(\delta).
\]
Furthermore, for $\delta$ sufficiently small we clearly have $\Re(\lambda_3(\xi,\delta))<0$.

Next, we deal with the double root $\bar\lambda=0$ of \eqref{ev3}.  To this end,
we assume $|\bar\lambda|\leq \gamma/2$
and note that, from \eqref{ev3}, one has
$$
\displaystyle
\bar\lambda^2(\bar\lambda+\gamma+O(\delta+\bar\xi))=O(\bar\lambda\bar\xi+\bar\xi^2).
$$
Since $|\bar\lambda|\leq\gamma/2$, it follows that
$$
\displaystyle
\bar\lambda^2=O(\bar\lambda+\bar\xi^2)
$$
\noindent
hence $\bar\lambda=O(\bar\xi)$ or, equivalently, $\lambda=O(\xi)$.
Now, we return to \eqref{ev2} in order to determine the real
part of the roots $\bar\lambda_i(\delta,\xi), i=1,2$ that bifurcate from the double
root $\bar\lambda=0$.  To this end, note that dividing \eqref{ev2} by
$\delta\xi^2$ and setting $\lambda=\tilde\lambda\xi$, $\xi=\bar\xi\delta$ yields the equation
\begin{equation}\label{ev4}
\displaystyle
\bar\xi\prod_{j=1}^3(\tilde\lambda-i\alpha_j^0)+\gamma\prod_{k=1}^2(\tilde\lambda-i\beta_k^0)=O(\bar\xi \delta+\delta).
\end{equation}
Now, letting $\bar\xi,\delta\to 0$ in \eqref{ev4} yields
\[
\gamma\prod_{k=1}^2(\tilde\lambda-i\beta_k^0)=0,
\]
which by (S1) and the fact that $\gamma\neq 0$ has two isolate roots $\tilde\lambda_1=i\beta_1^0$ and $\tilde\lambda_2=i\beta_2^0$.
Applying the implicit function theorem to \eqref{ev4} in a neighborhood of $(\tilde\lambda,\bar\xi,\delta)=(\tilde\lambda_j,0,0)$ for $j=1,2$,
implies the existence of two roots $\{\tilde\lambda_j(\bar\xi,\delta)\}_{j=1,2}$, defined for $\bar\xi,\delta$ sufficiently small,
bifurcating from the $\tilde\lambda_j$ which are smooth
functions of $\delta$ and $\bar\xi$ and can be expanded as
$$
\begin{array}{ll}
\displaystyle
\tilde\lambda_1(\xi,\delta)=i\beta_1^0+\frac{\bar\xi}{\gamma}\frac{\prod_{j=1}^3(\beta_1^0-\alpha_j^0)}{\beta_0^1-\beta_0^2}+O(\delta+\delta\bar\xi),\\
\displaystyle
\tilde\lambda_2(\xi,\delta)=i\beta_2^0+\frac{\bar\xi}{\gamma}\frac{\prod_{j=1}^3(\beta_2^0-\alpha_j^0)}{\beta_0^2-\beta_0^1}+O(\delta+\delta\bar\xi).
\end{array}
$$
\noindent
Since $\beta_j^0\in\RM$ by (S1), it follows that $\Re(\lambda_j(\xi,\delta))$, $j=1,2$, are even functions
of $\xi$ that can be expanded as
\begin{equation}\label{ev5}
\displaystyle
\Re(\lambda_1(\xi,\delta))=\frac{\xi^2}{\gamma\delta}\frac{\prod_{j=1}^3(\beta_1^0-\alpha_j^0)}{\beta_0^1-\beta_0^2}+O(\xi^2),\quad
\Re(\lambda_2(\xi,\delta))=\frac{\xi^2}{\gamma\delta}\frac{\prod_{j=1}^3(\beta_2^0-\alpha_j^0)}{\beta_0^2-\beta_0^1}+O(\xi^2)
\end{equation}
for $\xi,\delta$ sufficiently small.
Under conditions (S1),(S2),(S3), one proves $\Re(\lambda_i(\xi,\delta))<0$ for $|\xi|\leq\epsilon_0\delta$ and $0<\delta\leq\eta_0$ for $\eta_0,\epsilon_0$ sufficiently small. This completes the proof of the lemma
\end{proof}

In summary, it follows from Lemma \ref{lem_int} and Lemma \ref{lem_wh} that no further assumptions
than the subcharacteristic conditions (S1), (S2), and (S3) are necessary for the spectral
stability of the underlying periodic wave train to long-wavelength perturbations.
Furthermore, conditions (S1)-(S3) are validated by the numerical analysis in \cite{BN} for all
``near-KdV" profiles described in Proposition \ref{p:kdvsolnexpand} with periods 
$X\in[8.49,26.17]$,
demonstrating that the formal analysis conducted in \cite{BN} is indeed sufficient to conclude, up to machine error
the spectral stability of a given wave train $U=U_\delta$, $0<\delta\ll 1$.
Finally, we also note that the analysis of Sections \ref{sec3} and \ref{sec4} demonstrate, up to machine error,
that the wave trains of \eqref{kdv-ks}
which are found to be spectrally stable by the analysis of \cite{BN} are nonlinearly 
stable, in the sense defined in \cite{BJNRZ1}.

Next, we justify our terminology in referring to (S1)-(S3) as the subcharacteristic conditions by considering the formal Whitham averaged system of \eqref{kdv-ks} about a given wave train $U$ in the singular limit $\delta\to 0^+$.
Actually, this is precisely on the basis of the following formal discussion that the subcharacteristic conditions (S1)-(S3) were first conjectured to play a major in the small-Floquet stability of "near-KdV" waves \cite{NR2}.

\subsection{Whitham's modulation equations}\label{whitham}

It is now a classical result that Whitham's modulation equations for periodic
waves of conservation laws provide an accurate description of the spectral
curves at the origin, i.e. of the stability of a given wave train to weak large-scale perturbations
Let us mention here the work \cite{Se} in the general case,
\cite{NR1} for shallow water equations, \cite{NR2} for KdV-KS 
either for fixed $\delta>0$ or in the KdV limit 
and \cite{JZB,JZ1} for the generalized Korteweg-de Vries equation.  When considering \eqref{kdv-ks} in the
singular limit $\delta\to 0^+$, however, 
even the formal derivation 
of such a connection is more involved. In particular, we note that it is not sufficient to simply let $\delta\to 0$ in the modulation equations derived for \eqref{kdv-ks} with $\delta>0$ fixed. Instead, in this singular limit the introduction of a new set of modulation equations is required
\cite{NR2}. 
In this section, we 
recall the derivation of 
the appropriate modulation equations in this singular limit and
emphasize in which way the previous analysis 
%
demonstrates their connection with the spectrum at the origin of the linearized operator about a given wave train. In particular, 
the structure of the modulation equations 
will justify our terminology referring to conditions (S1), (S2), and (S3) as the ``subcharacteristic'' conditions.

Recall that the KdV-KS equation reads
\begin{equation}\label{ks_f}
\displaystyle
\partial_t u+\partial_x(\frac{u^2}{2})+\partial_x^3u+\delta(\partial^2_x u+\partial_x^3 u)=0.
\end{equation}
\noindent
Reproducing \cite{NR2}, we derive 
the Whitham modulation equations about a given periodic wave train of \eqref{ks_f} in the singular limit $\delta\to 0^+$.  To this end, we introduce the slow coordinates 
$(X,T)=(\varepsilon x,\varepsilon t)$, $\varepsilon\ll 1$,  
set $\delta=\bar\delta\varepsilon$ with $\bar\delta\in(0,\infty)$, and note that in the slow $(X,T)$ variables
equation \eqref{ks_f} reads
\begin{equation}\label{ks_fr}
\displaystyle
\partial_T u+\partial_X(\frac{u^2}{2})+\varepsilon^2\partial_X^3 u+\bar\delta(\varepsilon^2\partial_X^2 u+\varepsilon^4\partial_X^4u)=0.
\end{equation}
\noindent
Following \cite{Se},
we search for an expansion of $u$, solution of \eqref{ks_fr}, in the form
\begin{equation}\label{ans}
\displaystyle
u(X,T)=U^{(0)}(\frac{\phi(X,T)}{\varepsilon};X,T)+\varepsilon U^{(1)}(\frac{\phi(X,T)}{\varepsilon},X,T)+O(\varepsilon^2)
\end{equation}
\noindent
with $U^{(i)}(y;X,T)$ $1$-periodic in $y$.  Notice then that the local period of oscillation of $u^0$ in the variable
$y$ is $\eps/\partial_X\phi$, where we assume the unknown phase {\it a priori} satisfies the condition $\partial_X\phi\neq 0$.
By inserting this ansatz into \eqref{ks_fr} and collecting $O(\varepsilon^{-1})$ terms, one finds
\begin{equation}\label{kdv_w-1}
\displaystyle
\Omega\partial_y U^{(0)}+\kappa U^{(0)}\partial_y U^{(0)}+\kappa^3\partial_y^3U^{(0)}=0,
\end{equation}
\noindent
where $\Omega=\partial_T\phi$ and $\kappa=\partial_X\phi$.  Equation \eqref{kdv_w-1} is recognized as the
traveling wave ODE for the KdV equation \eqref{kdv} in the variable $\kappa y$ with wave speed 
$-\Omega/\kappa$. As such, equation \eqref{kdv_w-1} has a solution provided $\Omega=-\kappa c_0(u_0,\kappa,k)$ where now $u_0,\kappa$ and $k$ are considered as functions of the slow variables$(X,T)$.  In this case, the solutions of \eqref{kdv_w-1} can be expressed as
\begin{equation}\label{wh_o1soln}
\begin{array}{ll}
\displaystyle
U^{(0)}(y,X,T)=U_0(y,u_0,\kappa,k)=u_0+12k^2\kappa^2{\rm cn}^2(\kappa y,k),\\
\displaystyle
c_0(u_0,\kappa,k)=u_0+8\kappa^2k^2-4\kappa^2.
\end{array}
\end{equation}
\noindent
In what follows, we derive a system of ``modulation equations" describing the evolution of $(u_0,\kappa,k)$
as functions of the slow variables $(X,T)$.  One such equation comes from noticing that 
the compatibility condition $\partial_T\kappa=\partial_X\Omega$ yields the equation 
\begin{equation}\label{eq_wn}
\displaystyle
\partial_T \kappa+\partial_X(\kappa c_0(u_0,\kappa,k))=0
\end{equation}
for the local wave number $\kappa$.

To find other modulation equations we continue the above expansion and note that collecting the $O(1)$ terms
yields an equation of the form $L_{KdV}U^{(1)}=\dots$ where $L_{KdV}$ is the operator describing the linearized evolution of the KdV equation about $U^{(0)}$. Since the kernel of the adjoint of $L_{KdV}$ is spanned by $1$ and $U^{(0)}$, solvability conditions and thus the needed extra equations will be obtained by averaging in $y$ against $1$ and $U^{(0)}$ the $O(1)$ equation. Yet this equation being of the form
\begin{equation}\label{kdv_w0}
\displaystyle
\partial_T U^{(0)}+\partial_X\frac{(U^{(0)})^2}{2}=\partial_y\left(\cdots\right).
\end{equation}
it follows, averaging it over a single period in $y$, that
\begin{equation}\label{eq_mass}
\displaystyle
\partial_T\langle U_0(\cdot;u_0,\kappa,k)\rangle+\partial_X\left<\frac{U_0^2}{2}(\cdot,u_0,\kappa,k)\right>=0
\end{equation}
must be satisfied, where here $\left<f\right>:=\int_0^1f(y)dy$. To obtain the other solvability condition in an easy way, let us first remark, following the method used in \cite{JZ1} to derive modulations equations for the generalized Korteweg-de Vries equation, that by multiplying \eqref{ks_f} by $u$ we obtain an equation of the form
\begin{equation}\label{kdv_en}
\displaystyle
\partial_t\left(\frac{u^2}{2}\right)+\partial_x\left(\frac{u^3}{3}-\frac{3(\partial_x u)^2}{2}\right)=\delta\left((\partial_x u)^2-(\partial_x^2u)^2\right)+\partial_x^2(\cdots).
\end{equation}
\noindent
This implies that equation \eqref{kdv_w0} multiplied by $U_0$ yields 
\begin{equation}\label{kdv_en0}
\displaystyle
\partial_T\left(\frac{U_0^2}{2}\right)+\partial_X\left(\frac{U_0^3}{3}-\frac{3(U_0')^2}{2}\right)=\bar\delta\left((U_0')^2-(U_0'')^2\right)+\partial_y(\cdots).
\end{equation}
\noindent
Averaging \eqref{kdv_en0} over a period in $y$ then provides the 
balance 
law 
\begin{equation}\label{eq_en}
\displaystyle
\partial_T\left<\frac{U_0^2}{2}\right>+\partial_X\left< \frac{U_0^3}{3}-\frac{3(U_0')^2}{2}\right>=\bar\delta\left(\langle(U_0')^2\rangle-\langle(U_0'')^2\rangle\right).
\end{equation}
\noindent
Together, the homogenized system (\ref{eq_wn},\ref{eq_mass},\ref{eq_en}) forms a closed system of three conservation laws with a source term, called the averaged Whitham modulation system, 
describing the evolution of the quantities $(u_0,\kappa,k)$ as functions of the slow variables $(X,T)$.

Again repeating \cite{NR2}, let us now comment on the previous system. 
As a first step in analyzing the modulation system (\ref{eq_wn},\ref{eq_mass},\ref{eq_en}), notice that the steady states are given by points $(u_0^{\star},\kappa^\star,k^\star)\in\RM^3$ such that 
\[
\left<(U_0')^2(\;\cdot\;,u_0^{\star},\kappa^\star,k^\star)\right>=\left<(U_0'')^2(\;\cdot\;,u_0^{\star},\kappa^\star,k^\star)\right>,
\]
where $U_0$ is given as in \eqref{wh_o1soln}, i.e. $U_0$ corresponds to periodic traveling waves of 
\eqref{ks_f} in the limit $\delta\to 0$.  Indeed, by Proposition \ref{p:kdvsolnexpand}, 
these are simply the cnoidal wave trains of the KdV equation that can be continued as solutions of \eqref{kdv-ks}.
Now, letting $\bar\delta\to 0$, corresponding to large scale perturbations with frequency/wave number
of order $\varepsilon\gg\delta$, in the homogenized system (\ref{eq_wn},\ref{eq_mass},\ref{eq_en})
yields the Whitham averaged system for the Korteweg-de Vries equation; see 
\cite{W,JZ1}.
%
As stated previously, the numerical results in Figure \ref{fig2} in Section \ref{s:num} below
demonstrate that for all KdV cnoidal wave trains considered here the Whitham averaged system
for \eqref{kdv} is strictly hyperbolic with eigenvalues
\[
\alpha_1(u_0,\kappa,k)<\alpha_2(u_0,\kappa,k)<\alpha_3(u_0,\kappa,k),\quad \forall (u_0,\kappa,k)\in\mathbb{R}^3.
\]
Furthermore, in the limit $\bar\delta\to\infty$, corresponding to a relaxation limit 
and large scale perturbations with frequency/wave number $\varepsilon\ll\delta$ 
we obtain 
the relaxed system
\begin{equation}\label{wh_rel}
\begin{array}{ll}
\displaystyle
\partial_T \mathcal{G}(k)+\partial_X\left(\mathcal{G}(k)c_0(u_0,\mathcal{G}(k),k)\right)=0,\\
\displaystyle
\partial_T\left< U_0(\cdot;u_0,\mathcal{G}(k),k)\right>+\partial_X\left<\frac{U_0^2}{2}(\cdot;u_0,\mathcal{G}(k),k)\right>=0,
\end{array}
\end{equation}
\noindent
where here $\kappa=\mathcal{G}(k)$ is given by the selection principal in Proposition \ref{p:kdvsolnexpand}.
Notice that this system may also be obtained directly from the Whitham averaged system of conservation laws for the
KdV-KS equation \eqref{kdv-ks}, derived in \cite{NR2} for fixed $\delta>0$ as
$$
\displaystyle
\partial_T\kappa+\partial_X(\kappa c_{\delta}(M,\kappa))=0,\quad \partial_T M+\partial_X\left<\frac{U_\delta^2}{2}(M,\kappa)\right>=0, 
\quad M=\left< U_{\delta}(M,\kappa)\right>,
$$
\noindent
in the limit as $\delta\to 0$.  It is now well established \cite{Se,NR2} that a necessary 
condition for spectral stability of periodic traveling waves 
under large scale perturbations is that system \eqref{wh_rel} be hyperbolic, i.e. have only real eigenvalues.
In our analysis from Section \ref{s:specevans}, however, we assume the stronger condition
that the modulation system \eqref{wh_rel} is \emph{strictly} hyperbolic with eigenvalues
\begin{equation}\label{eig_wh}
\displaystyle
\beta_1(u_0,k)<\beta_2(u_0,k);
\end{equation}
\noindent
this corresponds precisely to condition $(S1)$ in Lemma \ref{lem_kdv}. 
It clearly follows that in considering only the relaxed hyperbolic system 
\eqref{wh_rel}, obtained by simply letting $\delta\to 0$ in the
Whitham modulation equations for \eqref{kdv-ks} derived for fixed $\delta>0$ that some information is lost: namely,
in this particular limit we obtain no information regarding conditions (S2) and (S3).

To understand the roles of conditions (S2) and (S3),
we must consider rather the
full modulation system (\ref{eq_wn},\ref{eq_mass},\ref{eq_en}) derived in the singular
limit $\delta\to 0$. For the sake of clarity, let us write this 
system with the parameterization $(\kappa,M,E)$ with $M=\langle U \rangle$ corresponding to the spatial average
of $U$ over a period and $E=\langle U^2/2\rangle$; see \cite{JZB,JZ1} for a discussion on such a parameterization of periodic
wave trains of the KdV equation \eqref{kdv}.  

In this parameterization, 
the modulation system
(\ref{eq_wn},\ref{eq_mass},\ref{eq_en}) 
recovers the form obtained in \cite{NR2}
\begin{equation}\label{wh_kmE}
\displaystyle
\partial_T\kappa-\partial_X(\Omega(\kappa,M,E))=0,\ \partial_T M+\partial_X E=0,\  
\partial_T E+\partial_X Q(\kappa,M,E)=\bar\delta R(\kappa,M,E),
\end{equation}
where $\Omega(\kappa,M,E)=-\kappa c_0(\kappa,M,E)$, $Q=\langle U_0^3-3(U_0')^2/2\rangle$ and $R=\langle(U_0')^2-(U_0'')^2\rangle$.
In the context of relaxation theory it is a classical assumption to suppose that the
condition $\partial_E R(\kappa^\star,M^\star,E^\star)\neq 0$ is satisfied, ensuring that near the equilibrium state 
$(\kappa^\star,M^\star,E^\star)$ the equation $E(\kappa,M,E)=0$ defines $E$ implicitly
in terms of $(\kappa,M)$; in what follows, we assume that this condition holds.

Under this assumption,
the subcharacteristic condition $(S3)$  can be easily interpreted.
Indeed, linearizing the modulation system \eqref{wh_kmE} about the steady state $(\kappa^\star,M^\star,E^\star)$ 
and restricting to spatially homogeneous, i.e. $X$-independent, perturbations 
yields the equation 
\begin{equation}\label{linwh}
\partial_T\tilde\kappa=0,\quad \partial_T \tilde M=0,\quad \partial_T \tilde E=\bar\delta\left(\partial_E R^\star\tilde E+d_{\kappa,M}R^\star(\tilde\kappa,\tilde M)\right),
\end{equation}
where $R^\star=R(\kappa^\star,M^\star,E^\star)$. 
Considered as a constant coefficient equation in the slow variables $(X,T)$, the 
dispersion relation of \eqref{linwh} is then given by
\begin{equation}\label{disp0}
\displaystyle
\lambda^2(\lambda-\bar\delta\partial_E R^\star)=0.
\end{equation}
\noindent
From this, it is clear from our spectral analysis in Section \ref{s:specevans} that the condition $(S3)$ 
is equivalent to $\partial_E R^\star<0$.  We note that this condition is a standard assumption in the 
context of relaxation theory, and is equivalent to requiring that the manifold of solutions of $R(\kappa,M,E)=0$ 
is stable.  

Furthermore, the dispersion relation \eqref{disp0} implies that two spectral curves bifurcate from the origin
as one allows the period of the perturbations to vary, corresponding to stability or instability with respect to weak long-wavelength perturbations.  It is 
a classical result \cite{W,Yo} that a necessary condition for the stability of
the steady states of \eqref{wh_kmE} to such large-scale perturbations is given by the subcharacteristic
condition
\begin{equation}\label{sub_w}
\displaystyle
\alpha_1^\star\leq\beta_1^\star\leq\alpha_2^{\star}\leq\beta_2^{\star}\leq\alpha_3^{\star},
\end{equation}
where here the $\alpha_j^\star$ and $\beta_j^\star$ denote the functions $\alpha_j^0$ and $\beta_j^0$, respectively,
evaluated at the associated steady state.
Notice that in our analysis from Section \ref{s:specevans}, however, we assume the stronger condition that the inequalities
in \eqref{sub_w} are \emph{strict}, corresponding precisely to 
condition (S2).


In summary, we have just reviewed how conditions (S1), (S2), and (S3) were introduced in \cite{NR2} as the strict subcharacteristic conditions for the relaxation type Whitham modulation system \eqref{wh_kmE},
derived from \eqref{kdv-ks} in the singular limit $\delta\to 0$. As a byproduct of the analysis, carried out in Section \ref{s:specevans}, of the exact role of conditions (S1)-(S3), our present work have thus also rigorously validated the role of the modulation system \eqref{wh_kmE} in the determination of the presence of unstable spectrum near the origin for "near KdV" waves. Moreover, recall that the previous analysis also implies that validation of these conditions follows (up to machine error) from the numerical results of \cite{BN}.


\subsection{Numerical computation of subcharacteristic conditions}\label{s:num}

As 
proved in Section \ref{s:specevans} the subcharacteristic conditions (S1)-(S3) are sufficient to conclude
for $\delta>0$ sufficiently small the absence of unstable spectrum in a sufficiently small neighborhood of the origin. 
Moreover, we noticed that these conditions follow directly from the numerical calculations of \cite{BN}. 
Yet, to finish this section, we provide an independent
verification of the conditions 
(S1)--(S3) 
using the connection to the Whitham modulation system 
of \cite{NR2} reviewed 
in Section \ref{whitham}. In particular, we rely on the parameterization of the ``near-KdV" wave trains
of \eqref{kdv-ks} described in Proposition \ref{p:kdvsolnexpand}.  

It is well known that the Whitham modulation equations for the KdV equation \eqref{kdv} can be diagonalized by quantities referred to as Riemann invariants; see \cite{W}.  To describe this diagonalization and introduce the appropriate set of Riemann invariants, we first recall some properties concerning the parameterization of the KdV wave trains. To begin, notice that traveling wave solutions of \eqref{kdv} are solutions of the form $u(x,t)=u(x-ct)$ for some
$c\in\RM$, where the profile $u(\cdot)$ satisfies the equation 
\[
uu'-cu'+u'''=0.
\]
Integrating once, one finds the profile $u$ satisfies the Hamiltonian ODE
$$
\displaystyle
u''+\frac{u^2}{2}-cu=a,
$$
\noindent
for some constant of integration $a\in\R$, which can then be reduced to the form
of a nonlinear oscillator as 
\begin{equation}\label{kdvquad}
\frac{(u')^2}{2}=q-W(u;a,c),\quad W(u;a,c)=\frac{u^3}{6}-c\frac{u^2}{2}-au,
\end{equation}
where again $q$ denotes a constant of integration and $W$ represents the effective potential energy
of the Hamiltonian ODE \eqref{kdvquad}.  On open sets of the parameter space $(a,q,c)\in\RM^3$
the cubic polynomial $q-W(u;a,c)$ has positive discriminant so that there exist real numbers
$u_1\leq u_2\leq u_3$ such that
\[
q-W(u;a,c)=\frac{1}{6}(u-u_1)(u-u_2)(u_3-u).
\]
By elementary phase plane analysis, it follows that for such $(a,q,c)$ the profile ODE \eqref{kdvquad}
admits non-constant periodic solutions.  Moreover, by identifying powers of $u$ we find in this parameterization that
\begin{equation}\label{cdef}
c=\frac{u_1+u_2+u_3}{3},\quad a=-\frac{1}{6}(u_1u_2+u_1u_3+u_2u_3),\quad q=\frac{u_1u_2u_3}{6}.
\end{equation}
Using straightforward elliptic integral calculations, we find that the periodic solutions
of \eqref{kdvquad} can be written in terms of the Jacobi cnoidal function $\cn(x,k)$ as
\begin{equation}\label{kdef}
u(\xi)=u_2+(u_3-u_2)\cn^2\left(\sqrt{\frac{u_3-u_1}{3}}\xi,k\right),\quad \xi=x-ct,\quad k^2=\frac{u_3-u_2}{u_3-u_1}.
\end{equation}
In particular, notice that all solutions of \eqref{kdv} are of form \eqref{kdef} up to a Galilean shift
and spatial translation.
Letting $\displaystyle X=\frac{2\pi}{\kappa}$ denote the period of the above wave train, it follows again
by standard elliptic function considerations that $\kappa$ can be expressed as
\begin{equation}\label{kappadef}
\kappa=\frac{\pi}{K(k)}\sqrt{\frac{u_3-u_1}{3}}, 
\end{equation}
where here 
\be\label{Kdef}
K(k)=\int_0^1\frac{dx}{\sqrt{1-x^2}\sqrt{1-k^2x^2}}
\ee
denotes the complete elliptic integral of the first kind.

Furthermore, in terms of this parameterization we note that 
\[
\langle u\rangle=u_1+2(u_3-u_1)\frac{E(k)}{K(k)},\quad \left<\frac{u^2}{2}\right>=c\langle u\rangle+a,
\]
where here $\langle\cdot\rangle$ denotes the spatial average (in $\xi$) over a period $X$
and
\be\label{Edef}
E(k)=\int_0^1\frac{\sqrt{1-k^2x^2}}{\sqrt{1-x^2}}\;dx
\ee
denotes the complete elliptic integral of the second kind.

With this preparation, we can introduce the Riemann invariants $(\omega_1,\omega_2,\omega_3)$ 
for the KdV equation \eqref{kdv}, which are defined in terms of the $u_i$ as
\[
\omega_1=\frac{u_1+u_2}{2},\quad \omega_2=\frac{u_1+u_3}{2},\quad \omega_3=\frac{u_2+u_3}{2}.
\]
In terms of this parameterization, we have
\[
\begin{array}{lll}
\displaystyle
u(\xi)=\omega_1+\beta_3-\omega_2+2(\omega_2-\omega_1)\cn^2\left(\sqrt{\frac{2(\omega_3-\omega_1)}{3}}\xi,k\right),\quad 
  \kappa=\frac{2\pi}{K(k)}\sqrt{\frac{2(\omega_3-\omega_1)}{3}},\\
\displaystyle
c=\frac{\omega_1+\omega_2+\omega_3}{3},\quad k^2=\frac{\omega_2-\omega_1}{\omega_3-\omega_1},\quad  
\langle u\rangle=\omega_1+\omega_2-\omega_3+4(\omega_3-\omega_1)\frac{E(k)}{K(k)},\\
\displaystyle
\left<\frac{u^2}{2}\right>=c\langle u \rangle+a,\quad a=-\frac{1}{6}\left(2\beta_1(\omega_2+\omega_3-\omega_1)+(\omega_1+\omega_2-\omega_3)(\omega_1+\omega_3-\omega_2)\right)
\end{array}
\]
The Whitham modulation equations for the KdV equations can be diagonalized by the Riemann invariants $\omega_i$, in the sense that they can be written as
\[
\partial_T\omega_i+V_i(\omega_1,\omega_2,\omega_3)\partial_X\omega_i=0,
\]
where the characteristic velocities $V_i$ are given explicitly by
\[
V_i(\omega_1,\omega_2,\omega_3)=\frac{\partial_{\omega_i}(\kappa c)}{\partial_{\omega_i}\kappa}
\]
or, alternatively, as $V_i(\omega_1,\omega_2,\omega_3)=c+\left(\partial_{\omega_i}\ln(\kappa)\right)$.
Clearly, the characteristic velocities $V_i(\omega_1,\omega_2,\omega_3)$ correspond to the eigenvalues of the Whitham modulation equations for \eqref{kdv} about the periodic traveling 
wave
given in \eqref{kdef} associated to $(\omega_1,\omega_2,\omega_3)$.  
To describe these velocities more explicitly, we find it more 
convenient to parameterize the problem by the variables 
$\omega_1, \Delta=\omega_3-\omega_1$ and $k^2=(\omega_2-\omega_1)(\omega_3-\omega_1)^{-1}$. 
In terms of $(\omega_1,\Delta,k^2)$,  an elementary calculation shows that the characteristic velocities
$V_i$ can be expressed as $V_i(\omega_1,\omega_2,\omega_3)=c+\zeta_i$, where
$\zeta_i=\frac{2\Delta}{3}b_i(k)$ and
\[
b_1(k)=\frac{k^2K(k)}{E(k)-K(k)},\quad b_2(k)=\frac{k^2(1-k^2)K(k)}{(1-k^2)K(k)-E(k)},\quad b_3(k)=\frac{(1-k^2)K(k)}{E(k)},
\]
with $K(k),E(k)$ as in \eqref{Kdef}, \eqref{Edef}
denoting 
elliptic integrals of the first and second kind.
In Figure \ref{fig2} we plot the characteristic velocities in terms of the period $X(k)$ of the underlying
KdV wave train.  In particular, we see that for all $k\in(0,1)$ the characteristic velocities are
distinct, corresponding to 
satisfaction of \eqref{vap_kdv}
 i.e. 
to strict hyperbolicity of the associated Whitham modulation equation.

Next, we compute the eigenvalues of the relaxed Whitham modulation system \eqref{wh_rel},
which is also the limit as $\delta\to0$ of the Whitham modulation system associated
to \eqref{kdv-ks} for fixed $\delta>0$ \cite{NR2}. 
Recall from Proposition \ref{p:kdvsolnexpand} that we must restrict ourselves to those cnoidal waves of form \eqref{kdef} such that the selection principle $\kappa=\mathcal{G}(k)$
holds.  In terms of the $(\omega_1,k^2,\Delta)$ parameterization of the KdV Whitham system, this
modulation system restricted to the ``near-KdV" wave trains discussed in Proposition \ref{p:kdvsolnexpand} 
can be expressed as
\begin{equation}\label{wh-nearkdv}
\partial_T\kappa+\kappa\partial_X c=0,\quad \partial_T\langle u\rangle+ \langle u\rangle \partial_X c+\partial_X a=0,
\end{equation}
where
\[
\langle u\rangle=\beta_1+(k^2-1+4\frac{E(k)}{K(k)})\Delta,\quad a=-\frac{1}{6}\left(3\beta_1^2+2(k^2+1)\Delta\beta_1-(k^2-1)^2\Delta^2\right),
\]
and, recalling Proposition \ref{p:kdvsolnexpand},
\[
\kappa=\mathcal{G}(k),\quad \Delta(k)=\frac{3}{2}\left(\frac{K(k)\mathcal{G}(k)}{2\pi}\right)^2.
\]
To compute the eigenvalues of this relaxed modulation system, using the Galilean invariance of \eqref{kdv}
we require $\langle u\rangle=0$, which is equivalent to requiring
$\omega_1=-(k^2-1+4\frac{E(k)}{K(k)})\Delta(k)$.  This reduction thus leaves the elliptic modulus 
$k$ as the only parameter of the problem.   
It is then a lengthly but straightforward calculation to show that the eigenvalues $\beta_i^{\star} (k), i=1,2$ of 
\eqref{wh-nearkdv} are given by the roots of the polynomial equation
\begin{equation}\label{wh_relaxedev}
A(k)\lambda^2-B(k)\lambda+C(k)=0,
\end{equation}
where the coefficients are given by $A(k)=\mathcal{G}'(k)$,
\begin{eqnarray*}
B(k)&=&\mathcal{G}'(k)\left(k^2-1-\frac{k^2+1}{3}+\frac{4E(k)}{K(k)}\right)\Delta(k)\\
&&\;\;\;\;-\mathcal{G}(k)\left[\left(k^2-1-\frac{k^2+1}{3}+\frac{4E(k)}{K(k)}\right)\Delta(k)\right]',\nonumber\\
\displaystyle
C(k)&=&-\frac{2\mathcal{G}(k)\Delta(k)(2k^2-1)}{9}\left(2k\Delta(k)+(k^2+1)\Delta'(k)\right).\nonumber
\end{eqnarray*}
Notice that the roots of \eqref{wh_relaxedev} correspond to the eigenvalues $\beta_j^*$ considered earlier.
In Figure~\ref{fig2} we have plotted the  characteristic wave speeds 
$\{\alpha_i^*(X(k))\}_{i=1,2,3}$ and $\{\beta_j^*(X(k))\}_{j=1,2}$ as functions of 
the period $X(k)$ of the underlying wave train.  
From these numerics, it is clear that the subcharacteristic conditions (S1) and (S2) are satisfied
for all waves with period  $X\geq X_c$, where the critical period is $X_c\approx 8$. 
For 
$X<X_c$, condition (S2) is violated, corresponding to a sideband (modulational) instability
of the associated wave train.  This threshold is consistent with 
the one 
found in \cite{BN}. 
Furthermore, since the low-frequency stability conditions (S1)-(S3) are satisfied for all
periods $X\geq X_c$, we see also that the upper stability boundary $X\approx 26.17$ cannot
be associated with a sideband instability, again consistent with the observations of \cite{BN}.
%

\begin{figure}[htbp]
\begin{center}
\includegraphics[scale=0.7]{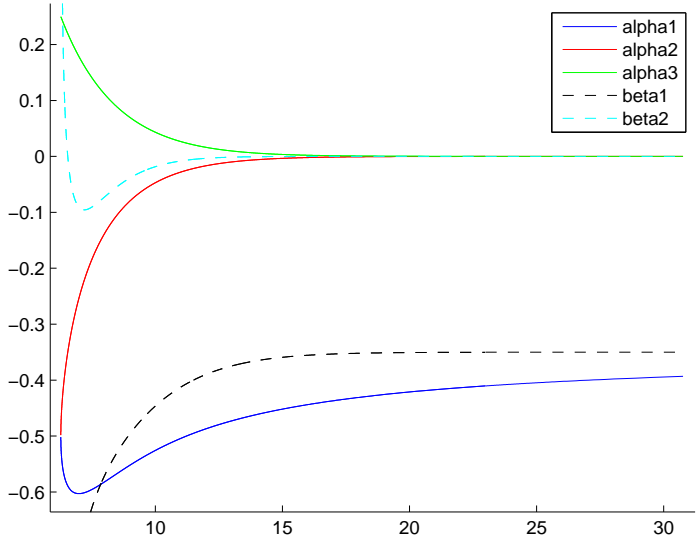}
\caption{\label{fig2} Here, we plot the characteristic velocities $\{\alpha_j(X(k))\}_{j=1}^3$ and $\{\beta_j(X(k))\}_{j=1}^2$ for the Whitham system for 
Korteweg-de Vries equation and the relaxed Whitham's system \eqref{wh-nearkdv}, respectively, as functions of the period $X(k)$ of the
underlying wave train.}
\end{center}
\end{figure}

\begin{figure}[htbp]
\begin{center}
\includegraphics[scale=0.6]{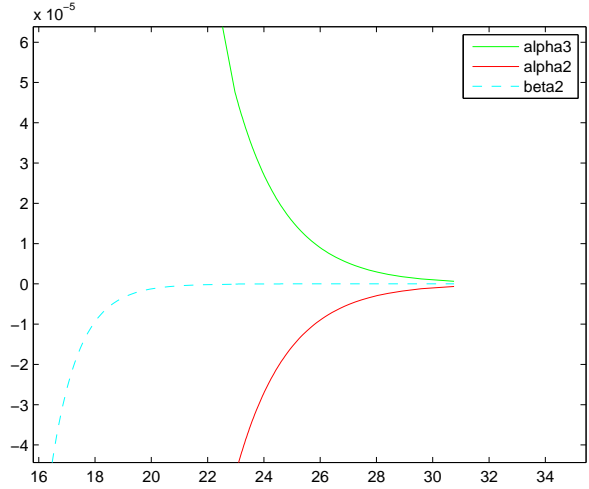}
\caption{\label{fig3} Here, we zoom in on the largest eigenvalues $\alpha_{2,3}(X(k))$ and $\beta_2(X(k))$ in Figure \ref{fig2}, noting
in particular that $\alpha_2(X(k))<\beta_2(X(k))<\alpha_3(X(k))$ for all periods $X(k)\geq X_c$ of the underling wave train.}
\end{center}
\end{figure}

Finally, we check the 
subcharacteristic condition (S3) and consider the spatially homogeneous perturbations (independent of the space variable). The Whitham's equations read
\begin{equation}\label{whith_hom}
\displaystyle
\partial_T\kappa=0,\quad \partial_T \langle u\rangle=0,\quad \partial_T\langle\frac{u^2}{2}\rangle=\bar\delta\left(\langle(u')^2\rangle-\langle(u'')^2\rangle\right),
\end{equation}
\noindent
where $u$ is defined by \eqref{kdef}. In this setting, we use $k,M=\langle u\rangle$ and $\Delta=u_3-u_1$ as parameters. One thus has
$$
\displaystyle
\kappa=\frac{\pi}{K(k)}\sqrt{\frac{\Delta}{3}},\quad \langle\frac{u^2}{2}\rangle=\frac{M^2}{2}-\frac{1}{6}P(k)\Delta^2,
$$
\noindent
with $\displaystyle P(k)=1-k^2+4(k^2-2)E(k)/K(k)+12(E(k)/K(k))^2$. Next, one can show that the source term is written in a simpler form
$$
\displaystyle
R(k,M,\Delta)=\bar\delta\left(\langle(u')^2\rangle-\langle(u'')^2\rangle\right)=r(k,M,\Delta)\left(\bar\Delta(k)-\Delta\right),
$$
with $r(k,M,\Delta)>0$ and $\bar\Delta$ given by:
$$
\displaystyle
\bar\Delta(k)=\frac{21}{20}\frac{2(k^4-k^2+1)E(k)-(1-k^2)(2-k^2)K(k)}{(-2+3k^2+3k^4-2k^6)E(k)+(k^6+k^4-4k^2+2)K(k)}.
$$
The steady states of \eqref{whith_hom} correspond to $\Delta=\bar\Delta(k)$. By linearizing \eqref{whith_hom} about a steady state $(k_*,M_*,\Delta_*=\bar\Delta(k_*))$ and searching for solutions that grow in time exponentially, one finds the dispersion relation
$$
\displaystyle
\Lambda^2(\Lambda-\Lambda_*)=0,
$$
with $\Lambda_*$ satisfying
$$
\displaystyle
\left((\frac{P(k_*)K'(k_*)}{3K(k_*)}+\frac{P'(k_*)}{12})\Delta_*^2\right)\Lambda_*=r(k_*,M_*,\Delta_*)\left(\frac{\Delta_*K'(k_*)}{K(k_*)}-\frac{\bar\Delta'(k_*)}{2}\right).
$$
\noindent
The subcharacteristic condition (S3) is satisfied if and only if $\Lambda_*<0$. One clearly sees 
that this condition is independent of $M_*$. In Figure \ref{fig4}, we have represented 
$\lambda_*=\Lambda_*/r(k_*,M_*,\bar\Delta_*)$ as a function of the period $X$. We 
clearly see that the 
subcharacteristic condition (S3) is always satisfied on the range of period $[2\pi, X_m]$ with $X_m\geq 30$.
In particular, $(S3)$ holds for all near-KdV wave trains with period $X\geq X_c$, corresponding to the low-frequency stability boundary, 
and $X\leq 26.17$, corresponding to the high-frequency stability boundary computed in \cite{BN}.

\begin{figure}[htbp]
\begin{center}
\includegraphics[scale=0.6]{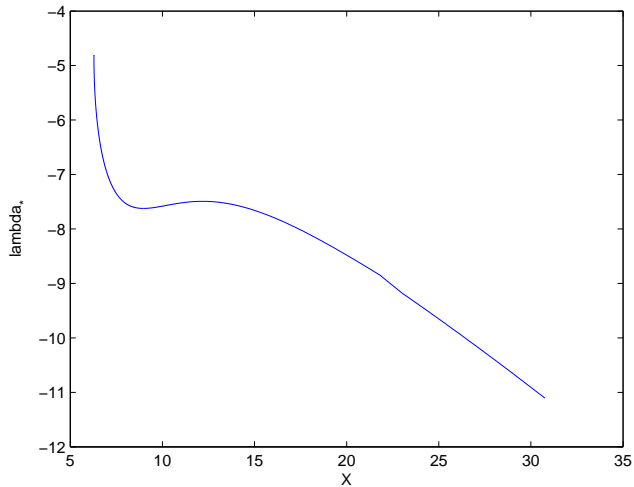}
\caption{\label{fig4} Here, we plot $\lambda_*(X(k))$ as a function of the period $X(k)$ of the underling wave train.}
\end{center}
\end{figure}

\medskip
{\bf Acknowledgement.} Thanks to Blake Barker for
the aid of his numerical Evans function computations carried out
for $\delta\ll 1$, some but not all of which appear
in \cite{BJNRZ1}.

\appendix

\section{Order one eigenvalues: the computations of \cite{BN}}\label{s:bn}

In this appendix, we supplement the analytical results of Section \ref{s:evexpand} by
describing in our own notation the numerical computations carried out in \cite{BN}
determining the sign of the real part of the $O(\delta)$ corrector $\lambda_1(\xi)$
for a fixed $\xi\in[-\pi/X,\pi/X)$ for nonzero eigenvalues $\lambda_0$ for KdV.
To this end, recall that for a fixed $\xi\in[-\pi/X,\pi/X)$ the $L^\infty(\RM)$ eigenvalues
$\lambda_\delta(\xi)$ can be expanded for $0<\delta\ll 1$ as in \eqref{blochexpand}, where we recall
that for $\lambda_0(\xi)\neq 0$ 
the real part of the $O(\delta)$ corrector can be found
from \eqref{rl1}, assuming $\int_0^{X}\hat{v}_{0,j}'\bar{\hat{v}}_{0,j}dx\neq 0$, as
\begin{equation}\label{corr}
\displaystyle
\Re(\lambda_1)=
\frac{
\Im\left(\int_0^{X}\hat{v}_{0,j}''\bar{\hat{v}}_{0,j}'-\hat{v}_{0,j}'''\bar{\hat{v}}_{0,j}''\right)}
{ \Im\left(\int_0^{X}\hat{v}_{0,j}'\bar{\hat{v}}_{0,j}\right) }=\frac{\langle \hat v_{0,j}'; \hat v_{0,j}''+\hat v_{0,j}''''\rangle }{\langle \hat v_{0,j}';\hat v_{0,j} \rangle};
\end{equation}
notice that this is precisely formula (54) on page 593, with $\Phi_0=\hat{v}_{0,j}$, for the $O(\delta)$ 
correction 
of non-zero KdV eigenvalues $\lambda_0$ found in \cite{BN}. 
Using Mathematica, the authors of \cite{BN} then numerically evaluate 
the quantity $\max_{\xi\in[-\pi/X,\pi/X)}\Re(\lambda_1(\xi))$,
which clearly must be non-positive to conclude stability.
The details of this calculation are as follows.  

First, denote 
$$
\displaystyle
\omega=\frac{\pi}{\kappa},\quad \omega'=\frac{K(\sqrt{1-k^2})\pi}{K(k)\kappa}.
$$
\noindent
Following the stability analysis for the KdV equation \eqref{kdv} presented in \cite{Sp}, 
the authors of \cite{BN} parameterize the eigenvalues and eigenfunctions
$\lambda_0$ and $\hat v_0$ and the Bloch wave number $\xi$ as 
\begin{equation}
\displaystyle
\hat v_0(x)=\frac{\sigma^2(x+i\omega'+\alpha)}{\sigma^2(x+i\omega')\sigma^2(\alpha)}e^{-2(x+i\omega')\zeta(\alpha)},\quad
\displaystyle
\lambda_0=-4\nu'(\alpha),\quad
\displaystyle
\xi=2i\left(\zeta(\alpha)-\frac{\alpha}{\omega}\zeta(\omega)\right),
\end{equation}
\noindent
where here $\sigma$ and $\zeta$ denote Weierstrass's sigma- and zeta-functions, respectively, 
$\nu(z)$ denotes the  Weierstrass elliptic function with periods $\omega=\frac{\pi}{\kappa}$ and $i\omega'$
where $\omega'=\frac{K(\sqrt{1-k^2})\pi}{\kappa K(k)}$. 
Notice 
that $\xi\in\R$ only if $\Re(\alpha)=n\omega, n\in\mathbb{N}$. 
In this case, the problem is parameterized by $\alpha$ 
and $k$, since $\kappa$ is determined by the selection criterion $\kappa=\widetilde{\mathcal{G}}(k)$ 
given by formula (34) on page 590 in \cite{BN}.  In \cite{BN}, 
the authors described the computations for $\alpha=n\omega+i\beta$ for $n=0,1$ and $\beta\in[0, 2\omega]$, 
claiming that the other cases $n\geq 2$ do not provide 
any new results.
Here, the parameter $k$ was restricted to
the interval $[0, 1-10^{-7}]$, which corresponds 
to periods $X=\frac{2\pi}{\kappa}$ lying approximately in the interval $[2\pi,10\pi]$.  
In order to evaluate the Weierstrass elliptic functions,
the usual theta functions are used:
$$
\begin{array}{ll}
\displaystyle
\Theta(z)=2\sum_{n=1}^\infty(-1)^{n+1}q_0^{(2n+1)^4/4}\sin\left((2n-1)\pi z/2K(k)\right),\\
\displaystyle
\Theta_1(z)=2\sum_{n=1}^\infty(-1)^{n+1}q_0^{(2n+1)^4/4}\cos\left((2n-1)\pi z/2K(k)\right),
\end{array}
$$
\noindent
with $q_0=\exp(-\pi K(\sqrt{1-k^2})/K(k))$. Then the various Weierstrass functions 
are
represented as 
\begin{align*}
\nu(z)&=e_1+\lambda\left(\frac{\Theta_1(z\sqrt{\lambda})\Theta'(0)}{\Theta_1(0)\Theta(z\sqrt{\lambda})}\right)^2,\\
\zeta(z)&=\zeta(\omega)\frac{z}{\omega}+\sqrt{\lambda}\frac{\Theta'(z\sqrt{\lambda})}{\Theta(z\sqrt{\lambda})},\\
\sigma(z)&=\frac{1}{\sqrt{\lambda}}\exp(\frac{\zeta(\omega)z^2}{2\omega})\frac{\Theta(z\sqrt{\lambda})}{\Theta'(0)},
\end{align*}
where $e_1=\nu(\omega), \lambda=\nu(\omega)-\nu(\omega+i\omega')$.

Using the above approach, it is found in \cite{BN} that the quantity $\max_{\xi\in[-\pi/X,\pi/X)}\Re(\lambda_1(\xi))$
is strictly negative for all KdV wave trains with periods in the interval $[8.49,26.17]$.
In particular, notice that from Figure \ref{fig2} the subcharacteristic conditions (S1)-(S3) hold
in this interval, as indicated in Section \ref{s:specevans}.
Furthermore, the left stability boundary corresponds to $\xi\approx 0$, hence 
to a sideband type instability;
as noted in the previous section, the right stability boundary does not.
For each $k$, and thus each period, the authors determine approximately the value $\xi_m$ where the functions $\lambda_1(k,\xi)$ 
take 
their
maximal values $\lambda_{1,m}(k)$. The points $\lambda_{1,m}(k)=0$ determine the boundaries 
of the stability region. 

As 
mentioned
throughout our analysis, it is important to note that the analysis of \cite{BN} 
a priori explores regions 
where the eigenvalues expand
as
$$
\displaystyle
\lambda(\delta,\xi)=\lambda_0(\xi)+\delta\lambda_1(\xi)+O(\delta^2),
$$
\noindent
and is 
thus limited only to some particular regions of the $(|\xi|,\delta)$ plane.  
In particular,  we stress that only the unveiling of the role of subcharacteristic conditions enables us to prove somewhat surprisingly that, though from the analysis of \cite{BN} it is not possible to conclude spectral stability, their numerical investigation is still sufficient to complete our analysis.

Finally, we note that another way of 
carrying out these computations would be to use instead the 
parameterization of eigenvalues and eigenvectors presented in \cite{BD}. In this case, one has
$$
\displaystyle
\lambda_0(\eta)=\pm 8i\sqrt{|\eta-\eta_1||\eta-\eta_2||\eta-\eta_3|},\quad \eta\in]-\infty,\eta_1]\cup[\eta_2, \eta_3],
$$
$$
\displaystyle
\xi=\frac{N\pi}{2K(k)}\pm\frac{8\sqrt{|\eta-\eta_1||\eta-\eta_2||\eta-\eta_3|}}{K(k)}\int_0^{K(k)}\frac{dy}{\eta-k^2+\dn(y,k)},
$$
$$
\displaystyle
\hat{v}_0(x)=\int_x^{x+X(k)}(\lambda_0(\eta)-\frac{U_0'(y)}{3})\exp\left(-\int_0^y \frac{\lambda_0(\eta)dz}{U_0(z)/3-c_0+\eta}\right)dy,
$$
\noindent
with $\eta_1=k^2-1$, $\eta_2=2k^2-1$, $\eta_3=k^2$ and $U_0$ the cnoidal wave given by setting $\kappa=\mathcal{G}(k)$
as defined in Proposition \ref{p:kdvsolnexpand}.

\end{document}